\documentclass[12pt]{amsart}
\usepackage{amsmath}
\usepackage{amssymb}
\usepackage{geometry}
\usepackage{graphicx}
\usepackage{amsthm}
\usepackage{setspace}
\usepackage{todonotes}

\usepackage{mathrsfs}        
\usepackage{bbm}

\allowdisplaybreaks

\newcommand{\diam}{{\text{diam}}}

\newcommand{\var}{{\text{var}}}

\usepackage{epstopdf}
\usepackage{amsthm}
\usepackage{enumerate}

\usepackage{dsfont}

\newtheorem{thm}{Theorem}
\newtheorem{lemma}{Lemma}
\newtheorem{proposition}{Proposition}
\newtheorem*{remark}{Remark}
\newtheorem{definition}{Definition}

\DeclareMathOperator{\interior}{int}

\providecommand{\floor}[1]{\left\lfloor \, #1 \, \right\rfloor}

\newcommand{\cA}{{\mathcal A}}

\title[Local escape rates for $\phi$-mixing dynamical systems]{Local escape rates for $\phi$-mixing dynamical systems}
  \date{\today}

\begin{document}

\begin{abstract}
We show that dynamical systems with $\phi$-mixing measures have
local escape rates which are exponential with rate $1$ at non-periodic
points and equal to the extremal index at periodic points.
We apply this result to equilibrium states on subshifts of finite type, Gibbs-Markov systems,
expanding interval maps, Gibbs states on conformal repellers and more
generally to Young towers and by extension to all systems that can be 
modeled by a Young tower.
\end{abstract}

\author{N Haydn}
\thanks{N Haydn, Department of Mathematics, University of Southern California,
Los Angeles, 90089-2532. E-mail: {\tt {nhaydn@usc.edu}}.} 
\author{F Yang}
\thanks{F Yang, Department of Mathematics, University of Oklahoma,
Norman,  73019-3103. E-mail: {\tt {fan.yang-2@ou.edu}}.}

\maketitle



{\em Acknowledgement:} The first named author was partially supported by Simons Foundation (ID 526571)

\tableofcontents
 
\section{Introduction}

Recently there has been increased interest in open dynamical systems
where a hole is placed in the space and points when they hit 
that hole are removed from the dynamic. This is like oil that has 
been spilled in the sea and is removed when it gets into contact 
with the coastline. The associated decay rate then describes the rate
at which mass is removed by hole over time. Clearly, a larger hole
has a larger escape rate and vice versa, a smaller hole will have a
smaller escape rate. Rates  of escape have 
been discussed for instance in \cite{AB10, BD07, CVdB, BY11, LMD, DY06}
and some other places. In some way such questions are 
similar to the statistics of hitting time where one looks at the scaled limiting 
distribution of hitting a shrinking neighbourhood. 

In this paper we consider localised escape rates which one 
obtains when the size of the hole shrinks to zero around a 
point. Since the escape rate will decrease as the hole shrinks
one scales the escape rate by the measure of the hole. 
The resulting limit is then called the localised escape rate.
In~\cite{BY11} the localised escape rate was considered for the 
doubling map on the interval and it was found to be equal to one
for non-periodic points. For periodic points it turned out to 
be given by the expansion rate of the map over a period. 
A similar result was shown in~\cite{FP} for conformal attractors
using restricted transfer operators on a shift space much in
the way as Hirata did in~\cite{Hirata1} for the entry times
distribution.

In this paper we show that the localised escape rate for 
$\phi$-mixing maps is exponential with rate one for 
non-periodic points and that for periodic points it is 
given by the extremal index. We then apply this result to
maps on metric spaces and derive the same limit 
in the case of balls. We show a similar result for maps
that allow for Young towers with exponential tails.
In the last section we give some examples, which are
(i) equilibrium states for Axiom A systems
where the holes are cylinder sets,
(ii) Gibbs-Markov systems
(iii) expanding maps of the interval which includes
 the case considered in~\cite{BY11,BDT}  and (iv)
 conformal maps where we recover the result of~\cite{FP}.
 For interval maps similar results were obtained in~\cite{BDT}
 using the transfer operator and the approximation results of~\cite{KL99}.

\section{Main results}

Let $T$ be a map on the space $\Omega$ and $\mu$ be a $T$-invariant probability measure 
on $\Omega$. We assume that there is a (measurable) partition $\mathcal{A}$ of 
$\Omega$ and denote by $\mathcal{A}^n=\bigvee_{j=0}^{n-1}T^{-j}\mathcal{A}$  its $n$th join. $\mathcal{A}^n$ is a partition of $\Omega$ and its elements are
called $n$-cylinders. For a point $x\in\Omega$ we denote by $A_n(x)\in\mathcal{A}^n$ 
the unique $n$-cylinder that contains the point $x$. 
We assume that $\mathcal{A}$ is generating, that is $\bigcap_nA_n(x)$ consists of the 
singleton $\{x\}$.

\begin{definition} (i) The measure $\mu$ is {\em left $\phi$-mixing} with respect to  $\mathcal{A}$
if
$$
|\mu(A \cap T^{-n-k} B) - \mu(A)\mu(B)| \le \phi(k)\mu(A)
$$
for all $A \in \sigma(\mathcal{A}^n)$, $n\in\mathbb{N}$ and  $B \in \sigma(\bigcup_j \mathcal{A}^j )$,
where $\phi(k)$ is a decreasing function which converges to zero as $k\to\infty$. Here $\sigma(\mathcal{A}^n)$ is the $\sigma$-algebra generated by $n$-cylinders.\\
 (i) The measure $\mu$ is {\em right $\phi$-mixing} w.r.t.\  $\mathcal{A}$
if
$$
|\mu(A \cap T^{-n-k} B) - \mu(A)\mu(B)| \le \phi(k)\mu(B)
$$
for all $A \in \sigma(\mathcal{A}^n)$, $n\in\mathbb{N}$ and $B \in \sigma(\bigcup_j \mathcal{A}^j )$,
where $\phi(k)\searrow0$.
\end{definition}

As has been shown by Abadi~\cite{Abadi01} that the measure of cylinder sets decreases exponentially
 fast, that is $\mu(A_n(x))\le \gamma^n$ for all $n$ and $x$ where $\gamma<1$.
 Usually a $\phi$-mixing measure is understood to be left $\phi$-mixing,
 but since it is an asymmetric property we could just as well consider the 
 right $\phi$-mixing property and obtain equivalent results. A measure that
 is left $\phi$-mixing is not necessarily right $\phi$-mixing and vice versa.
 However in our case the arguments are symmetric and we will restrict ourselves
 to demonstrate the proofs only in the left $\phi$-mixing case. The proofs for 
 the right $\phi$-mixing case are done with the obvious modifications.
 
 For a subset $U\subset \Omega$ we define the {\em entry/return time function}
 $\tau_U$ as the first time along its orbit that a point $x\in\Omega$ hits $U$,
 that is 
 $$
 \tau_U(x)=\inf\{j\ge1: T^jx\in U\}.
 $$
 If we consider all $x$ in $\Omega$ then it is the entry time and if we 
 restrict to $U$ then it is the return time. According to Poincar\'e's 
 recurrence theorem $\tau_U$ is finite on $U$ almost surely and 
 if $\mu$ is ergodic then $\int_U\tau_U(x)\,d\mu(x)=1$ by Kac's theorem.
 Clearly $\tau_U\ge\tau_{U'}$ if $U\subset U'$.
 We can now define the {\em escape rate} to $U$ by
 $$
 \rho_U=\lim_{t\to\infty}\frac1t|\log\mathbb{P}(\tau_U>t)|
 $$
 whenever the limit exists.
 Observe that if $U\subset U'$ then $\mathbb{P}(\tau_{U'}>t)\le\mathbb{P}(\tau_{U}>t)$
 and consequently $\rho_U\le\rho_{U'}$.
 For the partition $\mathcal{A}$ 
 we let $U_n\in\sigma(\mathcal{A}^n)$, $n=1,2,\dots$, be a sequence of 
 nested sets that contracts to a single point $x$ then we define by
 $$
 \rho(x)=\lim_{n\to\infty}\frac{\rho_{U_n}}{\mu(U_n)}.
 $$
 the {\em localised escape rate} if the limit exists. Part of this paper is to 
 show that under some reasonable assumptions on $U_n$ the limit does not 
 depend on the sequence $U_n$. For instance one can take $U_n=A_n(x)$.
 
 The return times statistics has been studied extensively beginning with
 Pitskel~\cite{Pit91}, Hirata~\cite{Hirata1} and others. For $\psi$-mixing measures Galves and Schmitt
 developed a method to show that the first entry time is exponentially distributed
 which subsequently was by Abadi~\cite{Abadi01,Abadi04} extended to $\phi$-mixing and then also
 to $\alpha$-mixing systems. It was found that entry times and return times are
 exponentially distributed for non-periodic points and that for 
 periodic points  the return times distribution has a point mass $1-\vartheta$
 at the origin which corresponds to the periodicity at the point.
 To be more precise, let $x\in\Omega$ be a periodic point with minimal period
 $m$, then 
 $$
 \vartheta(x)=\lim_{n\to\infty}\frac{\mu(U_n\cap T^{-m}U_n)}{\mu(U_n)}
 $$
 provided the limit exists. For simplicity we put $\vartheta(x)=0$ for 
 non-periodic points $x$.

For a subset $U_n\in\sigma(\mathcal{A}^n)$ we put for $j\le n$
$$
U_n^j=\begin{cases}A_j(T^{n-j}U_n)&\mbox{if $\mu$ is left $\phi$-mixing}\\
A_j(U_n)&\mbox{if $\mu$ is right $\phi$-mixing}\end{cases},
$$
where 
$$
A_j(U_n)=\bigcup_{B\in\mathcal{A}^j:\,B\cap U_n\not=\varnothing}B
$$
is the outer $j$-cylinder approximation of $U_n$.
Notice that in either case $U_n^j$ lies in $\sigma(\mathcal{A}^j)$ and 
is the best outer approximation of $U_n$ by unions of $j$-cylinders 
from the right or left respectively. In the first case we have $U_n \subset T^{-(n-j)}(U_n^j)$ and in the second case $U_n \subset U^j_n$ for all $j\le n$.

Let $U_n\in\sigma(\mathcal{A}^n), n=1,2,\dots$ be a sequence of neighbourhoods of a point $x\in\Omega$. 
If $x$ is periodic with minimal period $m$ then we put 
$U_{n,u}=\bigcap_{j=0}^uT^{-jm}U_n$  ($U_n=U_{n,0}$) and similarly 
$U_{n,u}^j=T^{n+mu-j}U_{n,u}$ in the left $\phi$-mixing case and 
$U_{n,u}^j=A_j(U_{n,u})$ in the right $\phi$-mixing case.
 In particular  $U_{n,u}\in\sigma(\mathcal{A}^{n+mu})$
and  $U_n^j, U_{n,u}^j\in\sigma(\mathcal{A}^{j})$.

We say $U_n, n=1,2,\dots$ is an {\em adapted neighbourhood system at $x$} if
(we use $\mathcal{O}^*$ to indicate a term whose implied constant is equal to $1$):
\begin{enumerate}
\item $U_n\in\sigma(\mathcal{A}^n)$;

\item $U_{n+1}\subset U_n$;

\item $\bigcap_{n}U_n=\{x\}$;

\item There exists $\gamma'>0$ so that $\mu(U_n^j)\lesssim j^{-\gamma'}, \forall j\le Kn$ if $x$ is 
non-periodic, and $\mu(U_{n,u}^j)\lesssim j^{-\gamma'}\,\forall j\le K(n+um)$ if $x$ is periodic
with minimal period $m$,
 for some $K\in(0,1]$.

\item If $x$ is periodic with minimal period $m$ then there exist $J(n)\in(0,1)$ so that
$J(n)n\to\infty$ as $n\to\infty$ and
 $\mu\!\left(\bigcap_{j=0}^kT^{-i_j}U_n\right)=\mu(U_{n,i_k/m})(1+\mathcal{O}^*(r_{n}))$
 for all $\vec{i} = (i_0,\dots,i_k)$ and all $k$, where $0=i_0<i_1<i_2<\cdots<i_k$, multiples of $m$,
  satisfy $i_k\le Jn$.
 The error term $r_{n}$ goes to zero as $n\to\infty$. 
 
\end{enumerate}

\vspace{2mm}
\begin{remark}
Observe that the $n$-cylinders $U_n=A_n(x)$ satisfy  conditions~(1)--(5).
In Property~(5), if we have $i_0 \ne 0$ then by invariance of $\mu$, this property still hold with $U_{n,(i_k-i_0)/m}$ instead of $U_{n,i_k/m}$. 
\end{remark}

\begin{thm} \label{escape.phi}
Let $\mu$ be a (left or right) $\phi$-mixing measure with respect to a generating 
partition $\mathcal{A}$ so that $\phi$ decays at least polynomially with a power
$p$.

Let $\{U_n\in\sigma(\mathcal{A}^n):\,n=1,2,\dots\}$ be an adapted neighbourhood
system satisfying Properties~(1) to~(5).  We consider the following two cases:

\noindent {\bf (I)} (polynomial case)  Assume there exist 
$1<\gamma'<\gamma''$ so that $n^{-\gamma''}\lesssim\mu(U_n)\lesssim n^{-\gamma'}$,
 in Property~(5)  $J(n)\gtrsim n^{-(1-\zeta)}$ for some 
$\zeta\in(0,1]$ and $\gamma'$ is the same as in Property~(4).
Assume $p> \max\left\{2,\frac{2\gamma''}{\zeta\gamma'}-1\right\}$.

\noindent {\bf (II)} (exponential case) Assume there exist  $0<\xi_1<\xi_2<1$
so that $\xi_1^n\lesssim\mu(U_n)\lesssim\xi_2^n$,  in Property~(5) $J(n)=J\in(0,1)$ 
is a constant and in Property~(4) it is sufficient to have $\gamma'>1$.
Assume $p>\frac8J\frac{\log\xi_1}{\log\xi_2}-1$.

Then
$$
\rho(x)=\lim_{n\to\infty}\frac{\rho_{U_n}}{\mu(U_n)}
=1-\vartheta(x)
$$
provided the limit 
$\vartheta(x)= \lim_{n\to\infty}\frac{\mu(U_{n}\cap T^{-m}U_n)}{\mu(U_n)}<\frac12$ exists if 
$x$ is periodic with minimal period $m$ (and $\vartheta(x)=0$ if $x$ is non-periodic).
\end{thm}

\begin{remark} If $T$ is invertible, then one has to define the $n$-join by
$\mathcal{A}^n=\bigvee_{j=[n/2]}^{[n/2]+n-1}T^{-j}\mathcal{A}$
and the statement of the theorem applies. The proof in that case requires only
minor adjustments.
\end{remark}

\begin{remark} The value $1-\vartheta(x)$ is often referred to as the extremal index.
\end{remark}

\begin{remark}
Note that $\frac{2\gamma''}{\zeta\gamma'}$ is always larger than $2$.
In the exponential case the lower bound for $p$ can be improved to the condition
$p>2\frac{1+J}J\frac{\log\xi_1}{\log\xi_2}-1$. For this see the remark at the end of the
proof of Lemma~\ref{lemma.periodic}. When the measure of $U_n$ is stretch exponential in $n$, a similar result can be obtained using the same method. 
\end{remark}

\begin{remark} The condition that $\vartheta<\frac12$ in the periodic case could well
be a  artifact of the proof which requires the convergence of a geometric
series in $\frac\vartheta{1-\vartheta}$. However, if $x$ is a periodic point of $T$,
then passing to an iterate $\hat{T}=T^p$ 
for some suitable integer $p$ will achieve that $x$ is still periodic for $\hat{T}$
and the new value of $\vartheta$ for the
map $\hat{T}$ will satisfy the requirement of being less than $\frac12$.
\end{remark}

 The existence of the limit defining $\vartheta$ for a period point with 
period $m$ implies that 
 $\lim_{n\to\infty}\frac{\mu(U_{n,u})}{\mu(U_n)}=\vartheta^u$. 
 Let us note that for $\psi$-mixing systems the limit 
 $\lim_{u\to\infty}\left(\frac{\mu(U_{n,u})}{\mu(U_n)}\right)^\frac1u$ always exists,
 is independent of $n$ and equals $\vartheta$ (see~\cite{HV09}).
 
\section{Proof of Theorem~\ref{escape.phi}}
 
 \noindent Denote by $\tau(U)=\min\{j\ge1: T^{-j}U\cap U\not=\varnothing\}$ the {\em period} of $U$.

 \begin{lemma} Let $\mathcal{A}$ be a (finite) generating partition of $\Omega$.
 Let $U_n\in\sigma(\mathcal{A}^n)$, $n=1,2,\dots$, so that $U_{n+1}\subset U_n\,\forall n$
 and $\bigcap_nU_n=\{x\}$.
 
 Then the sequence $\tau(U_n)$, $n=1,2,\dots$, is bounded if and only if
 $x$ is a periodic point.
 \end{lemma}
 
\begin{proof} Let us put $\tau_n=\tau(U_n)$ and notice that $\tau_{n+1}\ge \tau_n$
 for all $n$. Thus either $\tau_n\to\infty$ or $\tau_n$ has a finite limit $\tau_\infty$. 
 Assume $\tau_n\to\tau_\infty<\infty$. Then $\tau_n=\tau_\infty$ for all $n\ge N$, for some $N$,
 and thus $U_n\cap T^{-\tau_\infty}U_n\not=\varnothing$ for all $n\ge N$.
 Since the intersections $U_n\cap T^{-\tau_\infty}U_n$ are nested, i.e.\
 $U_{n+1}\cap T^{-\tau_\infty}U_{n+1}\subset U_n\cap T^{-\tau_\infty}U_n$, 
 for all $n\ge N$, we get
 $$
\varnothing\not=\bigcap_{n\ge N}(U_n\cap T^{-\tau_\infty}U_n)
 =\bigcap_{n\ge N}U_n\cap\bigcap_{n\ge N} T^{-\tau_\infty}U_n
 =\{x\}\cap\{T^{-\tau_\infty}x\}
 $$ 
 which implies that $x=T^{\tau_\infty}x$ is a periodic point.
 Conversely, if $x $ is periodic then clearly the $\tau_n$ are bounded by its period.
\end{proof}

\noindent Before we embark on the proof of the theorem, let us prove the following 
two lemmata about non-periodic points and periodic points.

\begin{lemma}\label{lemma.nonperiodic}~\cite{AV09}
Let $\mu$ be $\phi$-mixing with $\phi(k)=\mathcal{O}(k^{-p})$ for some $p>1$.
Let $U_n\in\sigma(\mathcal{A}^n)$, $n=1,2,\dots$, be sequence of sets
so that $\mu(U_n^j)=\mathcal{O}(j^{-\gamma'})$ for $j\le Kn$ for some $K\in(0,1)$ and $\gamma'>1$.

If $\tau(U_n)\to\infty$, then
$$
\frac{\mathbb{P}(\tau_{U_n}\le s_n)}{s_n\mu(U_n)}\to1^-
$$
for any sequence $s_n, n=1,2,\dots$ which satisfies $s_n\mu(U_n)\to0$.
\end{lemma}

\begin{proof} Let us drop the index $n$ and assume $U\in\sigma(\mathcal{A}^n)$.
One has the simple upper bound 
$\mathbb{P}(\tau_U\le s)= \mu\!\left(\bigcup_{j=1}^sT^{-j}U\right)\le s\mu(U)$.
In order to find a lower bound put $N_s=\sum_{j=1}^s\chi_U\circ T^j$ for the 
counting function for hitting $U$ up to time $s$. Clearly $\{\tau_U\le s\}=\{N_s\ge1\}$
and also $\mathbb{E}(N_s)=\mu(N_s)=s\mu(U)$. By Cauchy-Schwarz
$$
\left(\mu(N_s)\right)^2=\left(\mu(N_s\chi_{N_s\ge1})\right)^2
\le\mu(N_s^2)\mu(\chi_{N_s\ge1}^2)=\mu(N_s^2)\mathbb{P}(\tau_U\le s).
$$
Thus 
$$
\mathbb{P}(\tau_U\le s)\ge\frac{s^2\mu(U)^2}{\mu(N_s^2)},
$$
where we can write
$$
\mu(N_s^2)=\mu\!\left(\sum_{j=1}^s\chi_U\circ T^j\right)^2
=s\mu(U)+2\sum_{j=1}^s(s-j)\mu(\chi_U(\chi_U\circ T^j)).
$$
For $j=1,\dots,\tau(U)-1$ the terms in the sum are zero. 
For $j=\tau(U),\dots, 2n$ we use $U^{[j/2]}\in\sigma(\mathcal{A}^{[j/2]})$ smallest so that 
$T^{n-[j/2]}U\subset U^{[j/2]}$. Then by the $\phi$-mixing property 
$$
\mu(\chi_U(\chi_U\circ T^j))\le\mu(U\cap T^{-(j+n-[j/2])}U^{[j/2]})
\le\mu(U)(\mu(U^{[j/2]})+\phi(j/2)).
$$
For $j=2n+1,\dots,s$ the left $\phi$-mixing property yields
$$
\mu(\chi_U(\chi_U\circ T^j))=\mu(U\cap T^{-j}U)
\le\mu(U)(\mu(U)+\phi(j-n))\le\mu(U)(\mu(U)+\phi(j/2)).
$$
Hence (if $Kn<\tau(U)$ then the first sum in the second estimate
is equal to zero)
\begin{eqnarray*}
\mu(N_s^2)
&\le& s\mu(U)\!\left(1+2\sum_{j=\tau(U)}^{2n}\mu(U^{[j/2]})+\sum_{j=2n+1}^s\mu(U)
+\sum_{j=\tau(U)}^\infty\phi(j/2)\right)\\
&\le& s\mu(U)\!\left(1
+c_1\!\left(\sum_{j=\tau(U)}^{Kn} j^{-\gamma'}+\sum_{j=Kn+1}^{2n}n^{-\gamma'}\right)
+s\mu(U)
+c_2\sum_{j=\tau(U)}^\infty(j/2)^{-p}\right)\\
&\le& s\mu(U)\!\left(1+c_3\tau(U)^{-\gamma'+1}+s\mu(U)
+c_4\tau(U)^{-p+1}\right)
\end{eqnarray*}
where we used the estimate $\mu(U^{[j/2]})\le C[j/2]^{-\gamma'}$ (if $[j/2]\le Kn$
otherwise we use the upper bound $(Kn)^{-\gamma'}$) and that $\phi(k)\sim k^{-p}$.
Consequently
\begin{eqnarray*}
\frac{\mathbb{P}(\tau_U\le s)}{s\mu(U)}
&\ge&\frac1{1+c_3\tau(U)^{-\gamma'+1}+s\mu(U)
+c_4\tau(U)^{-p+1}}
\end{eqnarray*}
and therefore $\frac{\mathbb{P}(\tau_U\le s)}{s\mu(U)}\to1^-$ as $\gamma'>1$
 if we let $s\mu(U)\to 0$ and $\tau(U)\to\infty$, provided the limit defining $\vartheta$
 exists.
\end{proof}

\begin{lemma} \label{lemma.periodic}
Let $\mu$ be a $\phi$-mixing measure where $\phi$ decays at least polynomially
with a power $p>2$.

Let $x$ be a periodic point with minimal period $m$
and $\{U_n:n\}$ be an adapted neighbourhood system
with $J(n)$ so that $J(n)n\to\infty$ as $n\to\infty$,
then, if the limit $\vartheta$ exists and is less than $\frac12$,
$$
\lim_{n\to\infty}\frac{\mathbb{P}(\tau_{U_n}\le s(n))}{s(n)\mu(U_n)}=1-\vartheta
$$
for any sequence $s(n)\to\infty$ which satisfies $s(n)=o(\mu(U_n^\frac{J(n)n}4)^{\frac{1}2})$ and 
$s(n)\mu(U_n)\to0$ (recall $U_n^\frac{J(n)n}4\in\sigma(\mathcal{A}^\frac{J(n)n}4)$).
\end{lemma}

\begin{proof}
Now let $x$ be a periodic point with minimal period $m$.
By Bonferroni (also called the inclusion-exclusion formula due 
to de Moivre~\cite{deM}\footnote{thanks to M Rychlik for pointing out this reference})
$$
\mathbb{P}(\tau_{U_n}\le s)=\mu\!\left(\bigcup_{j=1}^s T^{-j}U_n\right)
=\sum_{j=1}^s\mu(T^{-j}U_n)+\sum_{\ell=1}^{s-1}(-1)^{\ell}M_{\ell+1},
$$
where
$$
M_{\ell+1}=\sum_{\vec{i}\in I_{\ell+1}}\mu(C_{\vec{i}}),
$$
and 
$$
I_{\ell+1}(s)=\left\{\vec{i}=(i_0,i_1,\dots,i_{\ell})\in\mathbb{N}^{\ell}:
0\le i_0<i_1<\cdots<i_{\ell}\le s\right\}.
$$
Here we use the notation 
$$
C_{\vec{i}}=\bigcap_{j=0}^{\ell}T^{-i_j}U_n.
$$

We will split the summation on $M_l$ into two parts: the principal part $P_\ell$ consists of $\vec{i}$'s that has a single cluster of periodic overlaps, that is, $i_\ell-i_0 \le Jn$; the error part $R_\ell$ are those $\vec{i}$'s with more than one cluster. Here $J=J(n)$ is given by Property~(5). 

Then $M_{\ell+1}=P_{\ell+1}+R_{\ell+1}$,
where the error term is
$$
R_{\ell+1}=\sum_{k=2}^{\ell}\sum_{\vec{i}\in\tilde{\mathcal{B}}_k}\mu(C_{\vec{i}})
$$
and (with $j_1=0$)

\begin{eqnarray*}
\tilde{\mathcal{B}}_k&=&\left\{\vec{i}\in I_{\ell+1}: \exists j_1<j_2<\cdots<j_k 
\mbox{ so that } i_{j_{q+1}-1}-i_{j_q-1}>Jn,\: i_{j_{q+1}-1}-i_{j_q}\le Jn\right.\\
&&\hspace{4cm}\left. \forall q=1,\dots,k-1, \mbox{ where }  j_{k+1}-1=\ell\right\},
\end{eqnarray*}
where the index $k\ge2$ denotes the number of clusters and where $J=J(n)$ is as
in Property~(5).
In other words, we start from the right most position of intersection, $i_\ell$, 
and parse $\vec{i}$ into $k$ clusters, each of which has length roughly $Jn$. 
Note that $i_{j_q}$ and $i_{j_{q+1}-1}$ are the head  and tail of the $q$th cluster respectively, 
with $i_{j_{q+1}-1}-i_{j_q}\le Jn$ form the clusters of short returns which have periodic
behavior. Note that the tail of the $k$th cluster is located at $i_{j_{k+1}-1} = i_\ell$.

For the principal terms which are characterised by a single cluster, we also write
$$
\tilde{\mathcal{G}}_{\ell+1}(s)=\left\{\vec{i}\in I_{\ell+1}: i_\ell-i_{0}\le Jn\right\}
$$
and $\mathcal{G}_{\ell+1}=\{\vec{i}\in\tilde{\mathcal{G}}_{\ell+1}: i_0=0\}$.
Moreover, with $\mathcal{B}_k=\left\{\vec{i}\in\tilde{\mathcal{B}}_k: i_0=0\right\}$
we see that $\mathcal{G}^c=\bigcup_k{\mathcal{B}}_k$.
With the first position fixed we obtain by the invariance of $\mu$,
\begin{equation}\label{Rl+1.estimate}
R_{\ell+1}\le s\sum_{k=1}^{\ell}\sum_{\vec{i}\in{\mathcal{B}}_k}\mu(C_{\vec{i}}).
\end{equation}
Let $\vec{i}\in\mathcal{B}_k$ and $j_2-1,\dots,j_{k+1}-1$ the positions of the tails of clusters,
i.e.\ $i_{j_{q+1}-1}-i_{j_q-1}>Jn$. 
If we put $\mathcal{C}_q=\bigcap_{j=j_q}^{j_{q+1}-1}T^{-(i_j-i_{j_q})}U_{n}$ for the $q$-th cluster
then $$C_{\vec{i}}=\bigcap_qT^{-i_{j_q}}\mathcal{C}_q.$$
In order to find an upper bound for 
$\mu(C_{\vec{i}})$ we use that by Property~(5) for each cluster
$$
\mu\!\left(\mathcal{C}_q\right)=(1+o(1))\mu(U_{n,u_q})
$$
$q=1,\dots,k-1$, 
with $u_q=\frac1m(i_{j_{q+1}-1}-i_{j_q})$ denoting the number of periodic overlaps in the
$q$th cluster. Evidently $\mathcal{C}_q\in\sigma(\mathcal{A}^{n+mu_q})$. To obtain gaps between clusters in order to apply the $\phi$-mixing property, let  $\eta=\frac{J}4$ and define the set
$$
\mathcal{C}_q'=T^{(1-\eta)(n+mu_q)}\mathcal{C}_q,
$$
$q=2,\dots,k$. 
By definition
$\mathcal{C}_q'\in\sigma(\mathcal{A}^{n_q})$, where  $n_q=\eta(n+mu_q)$ for $q\ge2$
and 
$$
\mathcal{C}_q \subset T^{-(1-\eta)(n+mu_q)}\mathcal{C}'_q.
$$
For $q=1$ we put $\mathcal{C}_1'=\mathcal{C}_1$.
As a result
$$
C_{\vec{i}}\subset\mathcal{C}_1\cap\bigcap_{q=2}^kT^{-(i_{j_q}+(1-\eta)(n+mu_q))}\mathcal{C}_q'.
$$
(keep in mind that in $\mathcal{B}_k$ we have $i_{j_1} = i_0 = 0$).
In this way we obtain gaps of lengths (note that $T^{-(i_{j_q}+(1-\eta)(n+mu_q))}\mathcal{C}_q'$ starts at $i_{j_q}+(1-\eta)(n+mu_q)$):

\begin{align*}
\hat\Delta_q &=i_{j_{q+1}}+(1-\eta)(n+mu_{q+1}) - (i_{j_q}+(1-\eta)(n+mu_q)) - \eta(n+mu_q)\\
&= i_{j_{q+1}} - i_{j_q} + m(u_{q+1}-u_q) - \eta(n+ mu_{q+1})\\
& = i_{j_{q+2}-1} - i_{j_{q+1}-1}- \eta (n+ mu_{q+1}),
\end{align*}
which we can write
\begin{equation}\label{Delta.hat}
\hat\Delta_q=i_{j_{q+2}-1}-i_{j_{q+1}-1}-\eta n-2\eta mu_{q+1}+\eta mu_{q+1}
=\Delta_q+\eta mu_{q+1},
\end{equation}
where ($J=J(n)$)
$$
\Delta_q=i_{j_{q+2}-1}-i_{j_{q+1}-1}-\eta n-2\eta mu_{q+1}
>Jn-\frac{J}4n-2\frac{J}4Jn>\frac{J}4n
$$
By the $\phi$-mixing property we thus obtain
\begin{equation}\label{cluster.mixing}
\mu(C_{\vec{i}})
\le\mu(\mathcal{C}_1)\prod_{q=2}^k\!\left(\phi(\Delta_q+\eta mu_q)+\mu(\mathcal{C}_q')\right).
\end{equation}

Since $n_q\le n$ (as $mu_q\le Jn$) we obtain the simple estimate by Properties~(4)
and~(5) (we assume $J\le K$)
$$
\mu(\mathcal{C}'_q)
=\mu(T^{(1-\eta)(n+mu_q)}U_{n,u_q})
\le \mu(U_n^{\eta n}).
$$
For the case $q=1$ (the very first term) we obtain by Property~(4) that 
$$
\mu(\mathcal{C}_1)=\mu(\mathcal{C}_1')=(1+o(1))\mu(U_{n,u_1}).
$$
Since 
$$
\mu(U_{n,u_1})
=\mu(U_n)\frac{\mu(\bigcap_{w=0}^{u_1}T^{-wm}U_{n})}{\mu(U_n)}\le c_1 \mu(U_n)\vartheta^{u_1}
$$
for some $c_1$, we conclude that 
\begin{eqnarray*}
\mu(C_{\vec{i}})
&\le &c_1 \mu(U_n)\vartheta^{u_1}
\prod_{q=2}^k\!\left(\phi\!\left(\Delta_q+\eta mu_q)\right)+\mu(U_n^{\eta n})\right).
\end{eqnarray*}
Since $k$ returns are long, the remaining $\ell-k$ returns must be short. Hence
$\sum_qu_q\ge \ell-k$.
Summing over all $\vec{i}\in\mathcal{B}_k$ yields the following estimate
$$
\sum_{\vec{i}\in\mathcal{B}_k}\mu(C_{\vec{i}})
\le\sum_{\sum_{q=1}^k\Delta_q\le s}\sum_{u=\ell-k}^{\frac{s}m}\sum_{\sum_{q=1}^ku_q=u}
c_1 \mu(U_n)\vartheta^{u_1}
\prod_{q=2}^k\!\left(\phi(\Delta_q+\eta mu_q)+\mu(U_n^{\eta n})\right).
$$
where $\Delta_q+\eta mu_q$ is the gap between the clusters.

For the $(\ell+1)$ remainder term $R_{\ell+1}$ we use~\eqref{Rl+1.estimate}
and get an estimate for the total error by interchanging summation between $k$ and $\ell$:
\begin{eqnarray*}
R&\le&s\sum_{\ell=1}^sR_{\ell+1}\\
&\le&c_2 s\mu(U_n)\sum_{k=1}^s\sum_{\ell=k}^s\sum_{\sum_{q=1}^k\Delta_q\le s}\sum_{u=\ell-k}^{\frac{s}m}\sum_{\sum_{q=1}^ku_q=u}
\vartheta^{u_1}\prod_{q=2}^k\!\left(\phi(\Delta_q+\eta mu_q)+\mu(U_n^{\eta n})\right)\\
&\le&c_2 s\mu(U_n)\sum_{k=1}^s\sum_{\sum_{q=1}^k\Delta_q\le s}\sum_{u_1=0}^\infty
\vartheta^{u_1}\prod_{q=2}^k\sum_{u_q=0}^s\!\left(\phi(\Delta_q+\eta mu_q)+\mu(U_n^{\eta n})\right)\\
&\le&c_3 s\mu(U_n)\sum_{k=1}^s\sum_{\sum_{q=1}^k\Delta_q\le s}
\prod_{q=2}^k\!\left(\phi^1(\Delta_q)+s\mu(U_n^{\eta n})\right)\\
\end{eqnarray*}
(as $\{\ell\in[k,s]\wedge u\in[\ell-k,\frac{s}m]\wedge \sum_{q=1}^ku_q=u\}
\subset \{u_q\in\mathbb{N}_0, q=1,\dots,k\}$; also notice that $n_q > \eta n$)
where $\phi^1(x)=\sum_{y=x}^\infty\phi(y)$ is the tail sum of $\phi$.
For the sum over the gaps $\Delta_q$ we use the fact that $\Delta_q\ge Jn/4$ 
and obtain
\begin{eqnarray*}
R&\le&c_2s \mu(U_n)\sum_{k=1}^s
\prod_{q=2}^k\sum_{\Delta_q=\frac{Jn}4}^s\!\left(\phi^1(\Delta_q)+s\mu(U_n^{\eta n})\right)\\
&\le&c_3 s\mu(U_n)\sum_{k=1}^s\left(\phi^2(Jn/4)+s^2\mu(U_n^{\eta n})\right)^{k-1}\\
&\le&c_5 s\mu(U_n)\left(\phi^2(Jn/4)+s^2\mu(U_n^{\eta n})\right)
\end{eqnarray*}
assuming that $\phi^2(Jn/4)+c_4s^2(Jn)^{-\gamma'}<\frac12$ say and where $\phi^2$
is the tail sum of $\phi^1$.

The principal terms are for the zeroth order $\sum_{j=1}^s\mu(T^{-j}U_n)=s\mu(U_n)$
and for higher order:
$$
P_{\ell+1}=\sum_{\vec{i}\in\tilde{\mathcal{G}}_{\ell+1}}\mu(C_{\vec{i}}),
$$
where
$$
\tilde{\mathcal{G}}_{\ell+1}=\left\{\vec{i}\in I_{\ell+1}: i_{\ell}-i_0\le Jn\right\}.
$$
Then by the invariance of $\mu$,
$$
P_{\ell+1}=\sum_{\vec{i}\in\tilde{\mathcal{G}}_{\ell+1}(s)}
\mu\!\left(U_n\cap \bigcap_{j=1}^\ell T^{-(i_j-i_0)}U_n\right)
=\sum_{r=\ell}^s\sum_{\vec{i}\in\mathcal{G}_{\ell+1}(r)}\mu(C_{\vec{i}}).
$$
For $\vec{i}\in\tilde{\mathcal{G}}_{\ell+1}(s)$ one has $i_{\ell}-i_0\le Jn$ and 
since the indices $\vec{i}$ in the summands have only a single cluster
which implies that $\mathcal{G}_{\ell+1}(r)=\mathcal{G}_{\ell+1}(s)$ for 
all $r\in(Jn,s]$ (recall that $\mathcal{G}_\ell = \{\vec{i} \in \tilde{\mathcal{G}}_\ell: i_0=0\}$) we thus obtain
$$
P_{\ell+1}=(1+o(1))sQ_{\ell+1}
$$
where 
$$
Q_{\ell+1}=\sum_{\vec{i}\in\mathcal{G}_{\ell+1}(s)}\mu(C_{\vec{i}}).
$$
For $\vec{i}\in\mathcal{G}_{\ell+1}$ each coordinate $i_j$ must be a multiple
of $m$, in particular we put $u(\vec{i})=\frac{i_\ell}m$. 
For a given $u$ ($u\le J\frac{n}{m}$) the cardinality 
$\left|\{\vec{i}\in\mathcal{G}_{\ell+1}: u(\vec{i})=u\}\right|$ is bounded above by 
$\binom{u-1}{\ell-1}$. Thus with $U_{n,u}=\bigcap_{w=0}^uT^{-wm}U_n$
we get by Property~(5):
$$
Q_{\ell+1}=\sum_{u=\ell}^{\ell\frac{n}{2m}}\binom{u-1}{\ell-1}\mu(U_{n,u})(1+o(1)).
$$
Since 
$\mu(U_{n,u})=(1+o(1))\mu(U_n)\vartheta^u$ we get 
\begin{eqnarray*}
Q_{\ell+1}
&=&(1+o(1))\mu(U_n)\sum_{u=\ell}^{\ell\frac{n}{2m}}\binom{u-1}{\ell-1}\vartheta^u\\
&=&(1+o(1))\mu(U_n)\left(\frac\vartheta{1-\vartheta}\right)^\ell-E_{\ell+1},
\end{eqnarray*}
where the correction for the tail sum
$$
E_{\ell+1}=(1+o(1))\mu(U_n)\sum_{u=\ell\frac{n}{2m}+1}^\infty\binom{u-1}{\ell-1}\vartheta^u
$$
can be estimated as
$$
|E_{\ell+1}|\le c_6(1+o(1))\mu(U_n)\left(\frac\vartheta{1-\vartheta}\right)^{\ell\frac{n}{2m}}.
$$

Finally, recall that the assumption $\vartheta<\frac12$ gives $\frac\vartheta{1-\vartheta}<1$, we get that 
$$
\mathbb{P}(\tau_{U_n}\le s)
=s(1+o(1))\mu(U_n)\sum_{\ell=0}^\infty(-1)^\ell\left(\frac\vartheta{1-\vartheta}\right)^\ell
+\mathcal{R}_n
=s\mu(U_n)(1+o(1))(1-\vartheta)+\mathcal{R}_n,
$$
where 
\begin{eqnarray*}
|\mathcal{R}_n|&\le& R+s(|E_{\ell+1}|)\\
&\le& c_7 s\mu(U_n)\left(\phi^2(Jn/4)+s^2\mu(U_n^{\eta n})\right)
+c_6 s\mu(U_n)\left(\frac\vartheta{1-\vartheta}\right)^{\ell\frac{n}{2m}}.
\end{eqnarray*}
Now we choose $s=s(n)=o(\mu(U_n^{\eta n})^{\frac{1}2})$ so that
$s(n)\to\infty$ and $s(n)\mu(U_n)\to0$ as $n$ goes to infinity.
If e.g.\ $\phi(k)= \mathcal{O}(k^{-p})$ for some $p>1$, then $\phi^2(k)=\mathcal{O}( k^{-p+2})$
goes to zero as $n\to\infty$ if $p>2$.  As a result  $\mathcal{R}_n=o(1)$. 
This then yields the result
$$
\rho(x)=\lim_{n\to\infty}\frac{\mathbb{P}(\tau_{U_n(x)}\le s)}{s\mu(U_n(x))}=1-\vartheta
$$
which proves the lemma since $\eta=\frac{J}4$.

Let us observe that we can choose $\eta<\frac{J}{1+J}$ arbitrarily close enough $\frac{J}{1+J}$
which implies that we can realise $s(n)$ to be of the order $\mathcal{O}(\mu(U_n^{\eta n}))$ 
for any $\eta<\frac{J}{1+J}$. Notice that this requires us in~\eqref{Delta.hat} to write
$\hat\Delta_q=\Delta_q+\hat\eta mu_{q+1}$ where $\hat\eta>0$ has to be sufficiently small.
\end{proof}

\vspace{3mm}

\begin{remark}
 In the case of right $\phi$-mixing, the parsing is done in a slightly different way: instead of starting first the right most intersection $i_\ell$, one need to start from $i_0$. If we denote by $j_q$ and $j_{q+1}-1$ the head and tail of the $q$th cluster, then one has $i_{j_{q+1}-1} - i_{j_q} < Jn$ and $i_{j_{q+1}} - i_{j_q} > Jn$. In (\ref{cluster.mixing}), $\mathcal{C}_1$ should be replaced by $\mathcal{C}_k$; the rest of the argument remains the same.
\end{remark}

\noindent The following lemma is very standard for mixing systems
(see e.g.~\cite{GS}).

\begin{lemma}\label{product.entry.time} Let $s,t>0$ and $U\in\sigma(\mathcal{A}^n)$.
 Then for all $\Delta<s/2$ one has
$$
\left| \mathbb{P}(\tau_U > s+t) - \mathbb{P}(\tau_U > t) \mathbb{P}(\tau_U > s)  \right| 
\le 2 (\Delta\mu(U) + \phi(\Delta-n))\mathbb{P}(\tau_U>t-\Delta).
$$
\end{lemma}

\begin{proof} We proceed in the traditional way splitting the difference into three parts:
$$
\big| \mathbb{P}(\tau_U > t+s) - \mathbb{P}(\tau_U > t) \mathbb{P}(\tau_U > s)  \big|
= I + II + III.
$$
In the first term we open up a gap of size $\Delta$. It is estimated as follows
\begin{eqnarray*}
I 
&=&\big| \mathbb{P}(\tau_U > t+s)-\mathbb{P}(\tau_U > t \cap \tau_U \circ T^{t+ \Delta} > s- \Delta)\big|\\
&\le& \mathbb{P}(\tau_U > t \cap  \tau_U \circ T^t\le \Delta)\\
&\le &\mathbb{P}(\tau_U > t-\Delta \cap  \tau_U \circ T^t\le \Delta)\\
&\le & \mathbb{P}(\tau_U > t- \Delta)(\mathbb{P}(\tau_U\le\Delta)+\phi(\Delta-n))\\
&\le & \mathbb{P}(\tau_U > t- \Delta)(\Delta\mu(U)+\phi(\Delta-n))
\end{eqnarray*}
where we used the $\phi$-mixing property.
Similarly for the third term in which we close the gap:
$$
III = \mathbb{P}(\tau_U > t) \mathbb{P}(s-\Delta<\tau_U\le s)
\le  \mathbb{P}(\tau_U > t) \Delta\mu(U) \le \mathbb{P}(\tau_U > t- \Delta)\Delta\mu(U).
$$
For the second term we use the $\phi$-mixing property. Since $U$ is a union of $n$-cylinders we get by the left $\phi$-mixing property:
$$
II=\big| \mathbb{P}(\tau_U > t \cap \tau_U \circ T^{t+ \Delta} > s- \Delta )  - \mathbb{P}(\tau_U > t) \mathbb{P}(\tau_U > s- \Delta)  \big| \le \phi(\Delta -n)\mathbb{P}(\tau_U > t)
$$
where $\Delta - n$ is the size of the gap in the mixing property.
Thus 
$II\le\phi(\Delta-n)\mathbb{P}(\tau_U>t-\Delta)$.

The three parts combined now prove the lemma.
\end{proof}

\begin{proof}[Proof of Theorem~\ref{escape.phi}.] Let $s$ and $\Delta<s$ be given. Then for any $t>0$ we can write 
$t=ks+r$ where $k=\floor{t/s}$ and $0\le r<s$. Let us assume $s$ is a multiple of $\Delta$ that is
$s=q\Delta$. By the last lemma we thus get for $j$ for which $jq\in\mathbb{N}$:
 $$
\mathbb{P}(\tau_U>js)
=\mathbb{P}(\tau_U>s)\mathbb{P}(\tau_U>(j-1)s)+\mathcal{O}^*(\delta)\mathbb{P}(\tau_U>(j-1-q^{-1})s)
$$
where $\delta=2(\Delta\mu(U)+\phi(\Delta-n))$. The notation $\mathcal{O}^*$ here
indicate that the involved constant is $\le1$, that is $g=\mathcal{O}^*(\delta)$ means $|g|\le \delta$.
Let $\eta=\frac{q}{q+1}$ and we want to show that
\begin{equation}\label{induction.estimate}
\mathbb{P}(\tau_U>ks)=\left(\mathbb{P}(\tau_U>s)+\mathcal{O}^*(\delta^\eta)\right)^{k-2}
\end{equation}
for $k\ge3$ and $kq\in\mathbb{N}$.
 To verify~\eqref{induction.estimate} for small values of $k$ note that 
 $\mathbb{P}(\tau_U>js)\le\mathbb{P}(\tau_U>s)$ for $j\ge1$ ($jq$ integer).
  This implies that for $3\le k\le4$, $kq\in\mathbb{N}$, one has
\begin{eqnarray*}
\mathbb{P}(\tau_U>ks)
&=&\mathbb{P}(\tau_U>s)\mathbb{P}(\tau_U>(k-1)s)+\mathcal{O}^*(\delta)\mathbb{P}(\tau_U>(k-1-q^{-1})s)\\
&\le&\mathbb{P}(\tau_U>s)^2+\mathcal{O}^*(\delta)\mathbb{P}(\tau_U>s)\\
&\le&\left(\mathbb{P}(\tau_U>s)+\mathcal{O}^*(\delta^\eta)\right)^{k-2}
\end{eqnarray*}
which verifies~\eqref{induction.estimate} for $k\in[3,4]$. The induction step for $k>4$, $kq\in\mathbb{N}$, is then
\begin{eqnarray*}
\mathbb{P}(\tau_U>js)
&=&\mathbb{P}(\tau_U>s)
\left(\mathbb{P}(\tau_U>s)+\mathcal{O}^*(\delta^\eta)\right)^{j-3}+\mathcal{O}^*(\delta)\left(\mathbb{P}(\tau_U>s)+\mathcal{O}(\delta^\eta)\right)^{j-3-q^{-1}}\\
&=&\left(\mathbb{P}(\tau_U>s)+\mathcal{O}^*(\delta^\eta)\right)^{j-3-q^{-1}}
\left(\mathbb{P}(\tau_U>s)\left(\mathbb{P}(\tau_U>s)+\mathcal{O}^*(\delta^\eta)\right)^{q^{-1}}
+\mathcal{O}^*(\delta)\right)
\end{eqnarray*}
Since $\delta=\delta^\eta\delta^\frac\eta{q}\le \delta^\eta\left(\mathbb{P}(\tau_U>s)+\mathcal{O}^*(\delta^\eta)\right)^{q^{-1}}$ one has
$$
\mathbb{P}(\tau_U>s)\left(\mathbb{P}(\tau_U>s)+\mathcal{O}^*(\delta^\eta)\right)^{q^{-1}}
+\mathcal{O}^*(\delta)
=\left(\mathbb{P}(\tau_U>s)+\mathcal{O}^*(\delta^\eta)\right)^{1+q^{-1}}
$$
thus completing the induction step:
$$
\mathbb{P}(\tau_U>js)=\left(\mathbb{P}(\tau_U>s)+\mathcal{O}^*(\delta^\eta)\right)^{j-2}.
$$
Hence~\eqref{induction.estimate} is valid for  all $k\ge3$, $kq$ integer.

We obtain 
$$
\frac1{ks}\log\mathbb{P}(\tau_U>ks)
=\frac{k-2}{ks}\log\left(\mathbb{P}(\tau_U>s)+\mathcal{O}^*(\delta^\eta)\right)
$$
and since $\mathbb{P}(\tau_U>s)=1-\mathbb{P}(\tau_U\le s)$ we conclude
$$
\lim_{k\to\infty}\frac1{ks}\log\mathbb{P}(\tau_U>ks)
=\frac1{s}\left(\mathbb{P}(\tau_U\le s)+\mathcal{O}(\delta^\eta)\right).
$$
Let $x$ be a periodic point and put  $U=U_n$.
Then
$$
\rho_{U_n(x)}=\lim_{k\to\infty}\frac1{ks}\log\mathbb{P}(\tau_{U_n(x)}>ks)
=\frac1{s}\left(\mathbb{P}(\tau_{U_n(x)}\le s)+\mathcal{O}(\delta_n^{\eta_n})\right)
$$
where $\eta_n=\frac{q_n}{1+q_n}$, $q_n=s(n)/\Delta_n$ (we assume this to be an integer) and 
\begin{equation}\label{delta_n}
\delta_n=2(\Delta_n\mu(U_n(x))+\phi(\Delta_n-n)). 
\end{equation}
Therefore
\begin{equation}\label{escape.error}
\rho(x)=\lim_{n\to\infty}\frac{\rho_{U_n}}{\mu(U_n)}
=\lim_{n\to\infty}\left(\frac{\mathbb{P}(\tau_{U_n}\le s(n))}{s(n)\mu(U_n)}
+\frac{\mathcal{O}(\delta_n^{\eta_n})}{s(n)\mu(U_n)}\right)
\end{equation}
In order to show that the limit of the second term on the RHS is zero in the case when 
$x$ is a periodic point we consider now the two cases: (I)~when $\mu(U_n)$ decays polynomially
 and (II)~when $\mu(U_n)$  decays exponentially.

\vspace{2mm}

\noindent {\bf (I)} Assume there exist $0<\gamma'<\gamma''$ so that 
$n^{-\gamma''}\lesssim\mu(U_n)\lesssim n^{-\gamma'}$. 
Let $\alpha\in(0,\gamma'/2)$, $\beta\in(0,1)$
and put $s=s(n)=[(Jn)^\alpha]$, $\Delta_n\sim s^\beta$ so that $q_n=s_n/\Delta_n$ is an integer.
Clearly $q_n\sim s(n)^{1-\beta}$ implies $\eta_n\to1$ as $n\to \infty$. Also, the fact that $s\to\infty$, 
implies that  $s(n)\mu(U_n)\to0$ and 
$s(n)=o(\mu(U_n^\frac{Jn}4)^\frac12)\lesssim (Jn)^\frac{\gamma'}2$ ($J=J(n)$)
of Lemmata~\ref{lemma.periodic} and~\ref{lemma.nonperiodic}
are satisfied.
If we assume that $\phi$ decays polynomially, i.e.\ 
$\phi(\ell)\lesssim \ell^{-p}$ for some $p>0$, then
$\delta_n=\mathcal{O}(s(n)^{\beta}\mu(U_n)+(Jn)^{-\alpha\beta p})$. 

Now we let $n$ go to infinity which implies that $s\to\infty$ and for the error term
\begin{eqnarray*}
\lim_{n\to\infty}\frac{\delta_n^{\eta_n}}{s(n)\mu(U_n)}
=\lim_{n\to\infty}\frac{\mathcal{O}(s(n)^{\beta\eta_n}\mu(U_n)^{\eta_n})}{s(n)\mu(U_n)}
+\lim_{n\to\infty}\frac{\mathcal{O}((Jn)^{-\alpha\beta p\eta_n})}{(Jn)^\alpha\mu(U_n)}.
\end{eqnarray*}
The first term on the RHS goes to zero as $\beta<1$ and $\eta_n$ converges to $	1$. 
Since by assumption $\mu(U_n)\ge c_1n^{-\gamma''}$ the second term on the RHS 
is bounded by 
$\mathcal{O}(n^{-(\zeta\alpha\beta p\eta_n+\zeta\alpha-\gamma'')})$
and converges to zero if $\alpha\beta p+\alpha-\frac{\gamma''}\zeta$ is positive (again $\eta_n\to1$).
This is achieved since by assumption $p>\frac{2\gamma''}{\zeta\gamma'}-1$
and  as $\alpha<\frac{\gamma'}2$  can be chosen close enough to $\frac{\gamma'}2$ 
and $\beta<1$ close enough to $1$.
Consequently
$$
\lim_{n\to\infty}\frac{\delta_n^{\eta_n}}{s(n)\mu(U_n)}=0.
$$

\vspace{2mm}

\noindent {\bf (II)} Now assume there exist $0<\xi_1<\xi_2<1$ so that 
$\xi_1^n\lesssim\mu(U_n)\lesssim\xi_2^n$. 
Let $\alpha, \beta<1$ and put $s(n)=\xi_2^{-\frac{J}8\alpha n}$ 
and $\Delta_n=s(n)^\beta$. The condition $s=o(\mu(U_n^\frac{Jn}4)^\frac12)$
of Lemma~\ref{lemma.periodic} is thus satisfied. 
We again use~\eqref{delta_n} and obtain that the first term on the RHS of 
$$
\frac{\delta_n^{\eta_n}}{s(n)\mu(U_n)}
\lesssim2\Delta_n^{-(1-\beta)\eta_n}+2\frac{\phi(\Delta_n-n)^{\eta_n}}{s(n)\mu(U_n)}
$$ 
obviously goes to $0$ as $\beta<1$. For the second term on the RHS we get 
$$
\frac{\phi(\Delta_n-n)^{\eta_n}}{s(n)\mu(U_n)}
\lesssim\frac{\xi_2^{\frac{J}8\alpha\beta p\eta_n n}}{\xi_2^{-\frac{J}8\alpha n}\xi_1^n}
=\xi_2^{n\left(\frac{J}8\alpha(\beta p\eta_n+1)-\frac{\log\xi_2}{\log\xi_1}\right)}
$$
converges to $0$ as $n\to\infty$ since $\phi$ decays at least polynomially 
with power $p>\frac8J\frac{\log\xi_2}{\log\xi_1}-1$
and $\alpha, \beta<1$ can be chosen arbitrarily close to $1$ and $\eta_n$ converges to $1$.
Again we conclude that the error term~\eqref{escape.error} is equal to zero.

\vspace{3mm}

\noindent Thus in both cases~(I) and~(II) we obtain  by an application of
 Lemma~\ref{lemma.periodic} that
$$
\rho(x)=\lim_{n\to\infty}\frac{\mathbb{P}(\tau_{U_n}\le s(n))}{s(n)\mu(U_n)}
=1-\vartheta.
$$
for any periodic point $x$ with $\vartheta <\frac12$.
If $x$ is a non-periodic point then by Lemma~\ref{lemma.nonperiodic} 
(there we need $\gamma'>1$)
we get that $\rho(x)=1$.
\end{proof}

\section{Escape rate for metric balls}\label{excape_metric_ball}

As an application of Theorem~\ref{escape.phi} we will indicate how one can obtain
the limiting distribution for metric balls for maps on metric spaces. 
We will still require that there be a generating partition with respect to which the measure
is right $\phi$-mixing. The balls will then be approximated by unions
of cylinders. 

Let $T$ be a map on a metric space $\Omega$ and let $\mathcal{A}=\{A_j:j\}$ be a generating finite
or countable infinite partition of $\Omega$, that is $\Omega=\bigcup_jA_j$ and
$A_j\cap A_i=\varnothing$ for $i\not= j$. 
As before we denote by $\mathcal{A}^n$ the $n$th joint of the partition.
Assume there is a $T$-invariant probability measure $\mu$ on $\Omega$. 

In order to obtain the escape rate for metric balls, we will approximate the balls $B_r(x)$ from inside and outside by unions of cylinders. For this purpose, for any given $r>0$ we will take $n=n(r)$ and put
$$
U_n^-=\bigcup_{A\in \mathcal{A}^n: \,A\subset B_r(x)}A,
\hspace{2cm}
U_n^+=\bigcup_{A\in \mathcal{A}^n:\,  A\cap B_r(x)\not=\varnothing}A
$$
the largest union of $n$-cylinders contained in $B_r(x)$ and the smallest union 
of $n$-cylinders that contains $B_r(x)$ respectively. The choice of $n(r)$ will be made clear in the proof of the next theorem.

The following result also has a formulation in the left $\phi$-mixing case:

\begin{thm}\label{escape.balls}
Let $\mu$ be an invariant measure on a Riemann manifold $\Omega$ with $C^{1+\alpha}$-map $T$
for some $\alpha>0$. 
Assume there
is a generating partition (finite or countably infinite) $\mathcal{A}$
with the following property where $x\in\Omega$:  \\
(i) $\mu$ is right $\phi$-mixing with rate $\phi(k)$ decaying at least polynomially with power $p$.\\
(ii) $\diam\,\mathcal{A}^n$ decays polynomially fast with power $\varpi$ as $n\to \infty$.\\
(iii) There exists $w>1$ such that 
$\frac{\mu(B_{r+r^w}(x))}{\mu(B_r(x))}\longrightarrow1$ as $r\to0^+$ for every $x$.\\
(iv) There exists $d>0$ so that $\mu(B_r(x))=\mathcal{O}(r^d)$ for all $r>0$ small enough and all $x$.\\

Assume $\frac{\varpi d}w>1$ and that Property~(5) is satisfied for the adapted neighbourhood systems $\{U_n^-\}$ and $\{U_n^+\}$and $J(n)\gtrsim n^{\zeta-1}$ for some $\zeta\in(0,1]$, then,
if $p>\frac{2-\zeta}\zeta\vee2$, 
$$
\rho(x)=\lim_{r\to0}\frac{\rho_{B_r(x)}}{\mu(B_r(x))}=\begin{cases}
1&\mbox{if $x$ is non periodic}\\
1-\vartheta&\mbox{if $x$ is periodic}\end{cases}
$$
for every $x\in\Omega$ provided the limit 
$\vartheta=\lim_r\frac{\mu(T^{-m}B_r(x)\cap B_r(x))}{\mu(B_r(x))}<\frac12$ exists for $x$ periodic
with minimal period $m$.
\end{thm}


\begin{proof} 
By assumption~(ii)  $\diam\,\mathcal{A}^n\le c_1n^{-\varpi}$
for some $c_1$. For every $r>0$ we choose $n=n(r)$ such that $r^w\ge n^{-\varpi}$.
This can be achieved by taking $n(r)=\left[2^{1/\varpi}r^{-w/\varpi}\right]+1$.
Fix $x$ and put
$$
U_n^-=\bigcup_{A\in \mathcal{A}^n: \,A\subset B_r(x)}A,
\hspace{2cm}
U_n^+=\bigcup_{A\in \mathcal{A}^n:\,  A\cap B_r(x)\not=\varnothing}A
$$
Note that in the case when $\diam\,\mathcal{A}^n$ decays exponentially then 
$n$ can be chosen proportional to $|\log r|$.

Assume $x$ is a periodic point with minimal period $m$ so that the limit defining
$\vartheta$ using metric balls exists.
In order to show that the same limits are obtained when using the 
approximations $U_n^\pm$,   
denote $r^\pm=r\pm r^w$
 (thus $B_{r-r^w}(x)\subset U_n^-$ and $U_n^+\subset B_{r+r^w}(x)$)
and note that by assumption~(iii)
$$
\frac{\mu(T^{-m}U_n^+\cap U_n^+)}{\mu(U_n^+)}
\le\frac{\mu(T^{-m}U_n^+\cap U_n^+)}{\mu(B_{r}(x))}
=\frac{\mu(T^{-m}B_{r^+}(x)\cap B_{r^+}(x))}{\mu(B_{r^+}(x))}(1+o(1))
$$
and similarly
$$
\frac{\mu(T^{-m}U_n^+\cap U_n^+)}{\mu(U_n^+)}
\ge\frac{\mu(T^{-m}B_{r}(x)\cap B_{r}(x))}{\mu(U_n^+)}
=\frac{\mu(T^{-m}B_{r}(x)\cap B_{r}(x))}{\mu(B_{r}(x))}(1+o(1)).
$$
The limits as $r\to0$ on the RHS of the last two inequalities exist by assumption and
equal $\vartheta$.  Hence 
$\vartheta=\lim_n\frac{\mu(T^{-m}U_n^+\cap U_n^+)}{\mu(U_n^+)}$.
In the same way one shows 
$\vartheta=\lim_n\frac{\mu(T^{-m}U_n^-\cap U_n^-)}{\mu(U_n^-)}$.

The properties~(1), (2) and~(3) of the neighbourhood systems  $\{U_n^-:n\}$
and $\{U_n^+:n\}$ are satisfied by construction. 
In order to satisfy Property~(4)  observe that 
$r^w=\mathcal{O}(n^{-\varpi})$ which implies that 
$V_j=\bigcup_{A\in\mathcal{A}^j:\,A\cap B_r(x)\not=\varnothing}A$ is the best approximation of $U^+_n$ by $j$-cylinders and
has diameter $\le r+\diam\,\mathcal{A}^j=\mathcal{O}(j^{-\varpi/w})$ for all
$j\le n$, i.e.\  $K=1$. By assumption~(iv) we conclude that $\mu(V_j)=\mathcal{O}(j^{-\varpi d/w})$.
Since $U_n^\pm\subset V_j$ we conclude that property~(4) is satisfied
with $\gamma'<\frac{\varpi d}w$ which by assumption is greater than $2$.
By Theorem~\ref{escape.phi} we now conclude that 
$$
\lim_{n\to\infty}\frac{\rho_{U_n^-}}{\mu(U_n^-)}
=\lim_{n\to\infty}\frac{\rho_{U_n^+}}{\mu(U_n^+)}=1-\vartheta,
$$
where we put $\vartheta(x)=0$ if $x$ is  not periodic.

In order to obtain the same limit for the balls $B_r(x)$ as $r$ goes to zero,
note that  $U_n^-\subset B_r(x)\subset U_n^+$ implies
$\rho_{U_n^-}\le\rho_{B_r(x)}\le\rho_{U_n^+}$ and consequently
$$
\frac{\rho_{U_n^-}}{\mu(U_n^+)}\le\frac{\rho_{B_r(x)}}{\mu(B_r(x))}\le\frac{\rho_{U_n^+}}{\mu(U_n^-)}.
$$
With $r^\pm=r\pm r^w$ one has 
$U_n^+\setminus U_n^-\subset B_{r^+}\setminus B_{r^-}$ and since by assumption 
$\frac{\mu(B_{r^+}\setminus B_{r^-})}{\mu(B_r)}\to0$ as 
$r\to0$ we also get that $\mu(U_n^-)=\mu(U_n^+)(1+o(1))$ as $n\to\infty$.
Hence $\lim_n\frac{\rho_{U_n^-}}{\mu(U_n^+)}=\lim_n\frac{\rho_{U_n^-}}{\mu(U_n^-)}$
and $\lim_n\frac{\rho_{U_n^+}}{\mu(U_n^-)}=\lim_n\frac{\rho_{U_n^+}}{\mu(U_n^+)}$.
 
For the values of $\gamma''$ and $\gamma'$ in Theorem~\ref{escape.phi} Case~(I) we take
 $\gamma''>\frac{\varpi d}w>\gamma'>1$ close enough to $\frac{\varpi d}w$ 
 so that $p>(\frac{2\gamma''}{\zeta\gamma'}-1)\vee2$ (recall that $p>2$ by assumption, so such $\gamma''$ and $\gamma'$ exist). 
 Since Property~(5) is satisfied by assumption, we thus can apply Theorem~\ref{escape.phi},
 Case~(I)
to the neighbourhood systems $\{U_n^\pm:n\}$
 and conclude that $\lim_{r\to0}\frac{\rho_{B_r(x)}}{\mu(B_r(x))}=1-\vartheta$.
\end{proof}

In the last theorem we require that Property~(5) is satisfied. In order to ensure
that this property is satisfied we will have to make additional assumptions
on the size of the cylinders which will have to decay exponentially and
the linearisation of the map at periodic points.  This is done in the following theorem.
We will require a somewhat non-standard annulus property. 

Let $x$ be a periodic point of $T$ of minimal period $m$ and assume that
the eigenvalues of the linearisation $DT^m(0)$ that are greater or equal to $1$ 
do not have generalised eigenvectors. That is, if $\lambda_1\ge\lambda_2\ge
\cdots\ge\lambda_u$ are the eigenvalues $\ge1$ then there are one-dimensional
eigenspaces $V(\lambda_j)$ on which $DT^m$ acts affinely.

Denote by $Q'_r(0)$ the product of $(-r,r)^u\subset \bigotimes_{j=1}^u V(\lambda_j)$
and the ball of diameter $r$ and centre $0$ in 
$\left(\bigotimes_{j=1}^u V(\lambda_j)\right)^\perp\subset T_x\Omega$.
Now put $Q_r(x)=\Phi_x(Q'_r(0))$, where $\Phi_x$ is the exponential map
at $x$.

\begin{thm}\label{escape.balls.exponential}
As in Theorem~\ref{escape.balls} let $T:\Omega\circlearrowleft$ be a $C^2$-map on 
a Riemann manifold $\Omega$ and $\mu$ a $T$-invariant probability measure.
Also, $\mathcal{A}$ is a generating partition of $\Omega$ so that:  \\
(i) $\mu$ is right $\phi$-mixing with rate $\phi(k)$ decaying at least polynomially with power  $p$.\\
(ii) $\diam\,\mathcal{A}^n\lesssim \eta^{n^\zeta}$ for some $\eta\in(0,1]$ and $\zeta\in(0,1]$.\\
(iii) There exists $w>1$ such that 
$\frac{\mu(B_{r+r^w}(x))}{\mu(B_r(x))},
\frac{\mu(Q_{r+r^w}(x))}{\mu(Q_r(x))}\longrightarrow1$ as $r\to0^+$ for every $x$.\\
(iv) There exists $d>0$ so that $\mu(B_r(x))=\mathcal{O}(r^d)$ for all $r>0$ small enough and all $x$.\\

If $p>\frac{2-\zeta}\zeta\vee2$ then
$$
\rho(x)=\lim_{r\to0}\frac{\rho_{B_r(x)}}{\mu(B_r(x))}=\begin{cases}
1&\mbox{if $x$ is non periodic}\\
1-\vartheta&\mbox{if $x$ is periodic}\end{cases}
$$
for every $x\in\Omega$ provided the limit 
$\vartheta=\lim_r\frac{\mu(T^{-m}B_r(x)\cap B_r(x))}{\mu(B_r(x))}<\frac12$ exists in the case when $x$ is a periodic
point with minimal period $m$, for which $DT^m(x)$ has no generalised eigenvectors
to eigenvalues larger than or equal to $1$.
\end{thm}

\begin{proof} In view of Theorem~\ref{escape.balls} we only have to verify 
Property~(5) at periodic points $x$ with period $m$ where $DT^m$ has 
simple eigenvalues.
 
By the smoothness of the map and the fact that $A=DT^m(x)$
has only single eigenvalues imply that Property~(5) is satisfied for the map
linearisation $A:T_x\Omega\circlearrowleft$,
that is $B'_r(0)\cap A^{-k}B'_r(0) =\bigcap_{j=0}^kA^{-j}B'_r(0)$ for any $k$, 
 where $B'_r(0)$ denotes the ball in $T_x\Omega$ of radius $r$ and centre $0$.
Denote by $\Phi_x$ the exponential map at $x$, then $B_r(x)=\Phi_x(B'_r(0))$.
If we let $y'\in B'_r(0)$ be so that $A^{-j}y'\in B'_r(0)$ then there exists a constant $c_1$
such that 
$$
\left|\Phi_x(A^{-j}y')-T^{-jm}\Phi_x(y')\right|
\le r\sup_{B'_r(0)}|D\Phi_x|
\le c_1 r^2
$$
as $\Phi_x$ is a local $C^2$-map. Let $w\in(1,2)$ and as in  Theorem~\ref{escape.balls} 
we require $n$ to be so that $\diam\,\mathcal{A}^n<r^w$ which 
is satisfied if $n> w\frac{\log r}{\log\eta}$.
Let us put $B^{(k)}_r(x)=\bigcap_{j=0}^kT^{-jm}B_r(x)$ and 
similarly $B'^{(k)}_r(x)=\bigcap_{j=0}^kA^{-j}B'_r(x)$. For any set $A\in\Omega$ denote by $B_r(A)$ the set 
$$
B_r(A) = \bigcup_{x\in A}B_r(x),
$$
then 
$$
\mathcal{D}=\Phi_x(B'^{(k)}_r(0))\triangle B^{(k)}_r(0)
\subset B_{r'}\!\left(\Phi_x\!\left(\partial B'^{(k)}_r(0)\right)\right)
$$
for some $r'\lesssim r^2$ since $d(T^{jm}y,\Phi_xA^j\Phi_x^{-1}y)\lesssim r^2$ 
for all $y\in B^{(k)}_r(x)$. Now notice that there are sets $Q_{\tilde{r}}$ of some radius 
 $\tilde{r}\ge |DT^m|_\infty^{-k}r$ inside the intersection 
 $\Phi_x(B'^{(k)}_r(0))$.

We put inside $\Phi_x(B_r'^{(k)}(0))$ sets $Q_{\tilde{r}_j}(x_j)$ 
 with centres $x_j$ and radii $\tilde{r}_j\gtrsim|DT^{i_k}(0)|^{-1}r\gtrsim |DT^{Jn}(0)|^{-1}r$,
  $j=1,2,\dots,R$,
 so that
 $$
 \mathcal{D}\subset\bigcup_j(Q_{\tilde{r}_j+\tilde{r}_j^w}(x_j)\setminus Q_{\tilde{r}_j-\tilde{r}_j^w}(x_j)).
 $$
Clearly we need
 $ \tilde{r}_j^w\gtrsim r'$, where $r'\lesssim r^2$. Since 
 $r^w\ge\eta^{n^\zeta}\ge\diam\,\mathcal{A}^n$, we get that $r^2\le\tilde{r}_j^w$ if
 the inequality $r^2\le|DT^m|_\infty^{-Jw\frac{n}m}r^w$ is satisfied.
 Consequently we must choose 
 $$
 J<m\frac{|\log\eta|}{\log|DT^m|_\infty}\frac{2-w}{w^2}\frac1{n^{1-\zeta}},
 $$
 where $\zeta$ is the same value to be used in Theorem~\ref{escape.balls}.
 We  therefore obtain by the annulus assumption 
 \begin{eqnarray*}
\frac{\mu(\mathcal{D})}{\mu(\Phi_x(B'^{(k)}_r(0)))}
&\le& \max_j\frac{\mu(Q_{\tilde{r}_j+\tilde{r}_j^{w}}\setminus Q_{\tilde{r}_j-\tilde{r}_j^{w}}(x_j))}
{\mu(Q_{\tilde{r}_j}(x_j))}
\longrightarrow0
 \end{eqnarray*}
 as $r$ goes to zero.
 
We have thus shown that if we have indices $0=i_0<i_1<\cdots<i_k\le Jn$ which are multiples of $m$,
then
$$
\mu\!\left(\bigcap_{j=0}^kA^{-i_j/m}B'_r(0)\right)
=\mu\!\left(\bigcap_{j=0}^kT^{-i_j}B_r(x)\right)\!\left(1
+o(1)\right)
$$
as $r\to0$. We have thus verified Property~(5) for the adapted neighbourhood system.

To check the conditions of Theorem~\ref{escape.balls}  we first note that
since by Assumption~(ii) the diameters decay superpolynomially, we can choose 
the value of $\varpi$ arbitrarily large. We can thus  apply Theorem~\ref{escape.balls}
to  conclude that $\lim_{r\to0}\frac{\rho_{B_r(x)}}{\mu(B_r(x))}=1-\vartheta$.
\end{proof}

\begin{remark}
Theorem~\ref{escape.balls.exponential} also applies to invertible maps. The only 
difference in that case is that the $n$-th join is given by
 $\mathcal{A}=\bigvee_{[n/2]\le j< [n/2]+n}T^{-j}\mathcal{A}$ and has to satisfy condition~(ii).
 \end{remark}

For an expanding map on a compact manifold Conditions~(ii), (iii) and~(vi) are met
for any invariant absolutely continuous measure.
In~(iv) the value of $d$ is equal to the dimension of the manifold.

 \section{Maps with Young's tower}\label{markov.towers}
 
In this section we consider differentiable maps on manifolds which can be modeled by Young's tower as constructed in \cite{Y2,Y3}.

We assume that  $T$ is a non-invertible, differentiable map of a
Riemannian manifold $M$. Assume that there is a subset $\Omega_0\subset M$ with the following properties:\\
(i) $\Omega_0$ is partitioned into disjoint sets $\Omega_{0,i}, i=1,2,\dots$ and
there is a {\em return time function} $R:\Omega_0\rightarrow\mathbb{N}$,  constant on
the partition elements $\Omega_{0,i}$,  such that $T^R$ maps $\Omega_{0,i}$ bijectively
to the entire set $\Omega_0$. We write $R_i=R|_{\Omega_{0,i}}$ for simplicity.\\
(ii) For $j=1,2,\dots, R_i-1$ put $\Omega_{j,i}=\{(x,j): x\in\Omega_{0,i}\}$
and define $\Omega =\bigcup_{i=1}^\infty\bigcup_{j=0}^{R_i-1}\Omega_{j,i}$. Note that $\{(x,0): x\in\Omega_{0,i}\}$ can be naturally identified with $\Omega_{0,i}$.
$\Omega$ is called the {\em Markov tower} or {\em Young's tower} for the map $T$. It has the associated partition
$\mathcal{A}=\{\Omega_{j,i}:\; 0\le j<R_i, i=1,2,\dots\}$ which typically is countably infinite.
The map $F: \Omega\to\Omega$ is given by
$$\
F(x,j)=\left\{\begin{array}{ll}(x,j+1) &\mbox{if } j<R_i-1\\
(\hat{T}x,0)&\mbox{if } j=R_i-1\end{array}\right.
$$
where we put $\hat{T}=T^R$ for the induced map on $\Omega_0$. If we denote by $\pi_\Omega:\Omega\to  M$, $\pi_\Omega((x,j))=T^jx$ then $\pi_\Omega$ semi-conjugates $F$ and $T$.\\
(iii) Non-uniformly expanding: there is $0<\kappa<1$ such that for all $x,y\in\Omega_{0,i}$, $d(\hat{T}x,\hat{T}y)>\kappa^{-1}d(x,y)$. Moreover, there is
$C>0$ such that $d(T^kx,T^ky)\le Cd(\hat{T}x, \hat{T}y)$ for all $x,y\in\Omega_{0,i}$ and $0\le k < R_i$.
\\
(iv) The {\em separation function} $s(x,y)$,  $x,y\in\Omega_0$, is defined as the largest positive
$n$ so that $(T^R)^jx$ and  $(T^R)^jy$ lie in the same sub-partition elements
for $0\le j<n$, i.e. $(T^R)^jx, (T^R)^jy\in\Omega_{0,i_j}$
for some $i_0,i_1,\dots,i_{n-1}$ while $(T^R)^jx$ and $(T^R)^jy$ belong to different $\Omega_{0,i}$'s.
We extend the separation function to all of $\Omega$
by putting $s(x,y)=s(F^{R-j}x,F^{R-j}y)$ for $x,y\in\Omega_{j,i}$.\\
(v) There is a finite given `reference' measure $\hat\nu$ on $\Omega_0$ which can be lifted by $F$ to a measure $\nu$ on $\Omega$: 
$$
\nu(A) = \sum_{i=1}^{\infty}\sum_{j=0}^{R_i-1} \hat\nu(F^{-j}(A\cap\Omega_{j,i})).
$$
We assume that the Jacobian $JF=\frac{d(F^{-1}_*\nu)}{d\nu}$
is  H\"older continuous in the following sense:
there exists a $\lambda\in(0,1)$ so that
$$
\left|\frac{JT^Rx}{JT^Ry}-1\right|\le C_2 \lambda^{s(\hat{T}x,\hat{T}y)}
$$
for all $x,y\in \Omega_{0,i}$, $i=1,2,\dots$.

The reference measure on $\Omega_0$ is often taken to be the Riemannian volume restricted to $\Omega_0$. If the return time $R$ is integrable with respect to  $\hat\nu$ then by~\cite{Y3}
Theorem~1 there exists an $\hat{T}$-invariant probability measure $\hat\mu$ on $\Omega_0$, which can be lifted to an $F$-invariant measure 
$\tilde{\mu}$ on $\Omega$; moreover, $\tilde{\mu}$ is absolutely continuous with respect to $\nu$. Then the pushed forward measure $\mu = (\pi_\Omega)_*\tilde{\mu}$ is a measure on $  M$ which is absolutely continuous with respect to the Riemannian volume.

Generally speaking, the conjugacy $\pi_\Omega$ need not be bijective, since the return time function $R$ may not be the first return to $\Omega_0$. However, when $R$ is the first return time, then $F$ and $T$ are indeed conjugate by $\pi_\Omega$. In this case the tower $(\Omega, F)$ is sometimes called a Rokhlin's tower.  

Now we are ready to state the theorem on the local escape rate for maps with Young's towers:

\begin{thm}\label{t.tower}
	Assume that $T$ is a $C^{2}$ map described above, such that $R$ is defined using the first return time and the reference measure is $\hat\nu=m|_{\Omega_0}$ where $m$ is the Lebesgue measure on $M$. Moreover, assume that $R$ has exponential tail: there exists $\lambda\in (0,1)$ such that 
	$$
	\hat\nu(R>n) \lesssim \lambda^n.
	$$ 
	Then we have
	$$
	\rho(x)=\lim_{r\to0}\frac{\rho_{B_r(x)}}{\mu(B_r(x))}=\begin{cases}
	1&\mbox{if $x$ is non periodic}\\
	1-\vartheta&\mbox{if $x$ is periodic}\end{cases}
	$$
	for every $x\in\interior(\Omega_0)$ provided the limit 
	$\vartheta=\lim_r\frac{\mu(T^{-m}B_r(x)\cap B_r(x))}{\mu(B_r(x))}<\frac12$ exists in the case when $x$ is a periodic
	point with minimal period $m$, for which $DT^m(x)$ has no generalised eigenvectors.
\end{thm}

\begin{remark}
	The theorem can be generalized to all $x\in M$ with $\tau_{\Omega_0}(x)<\infty$, such that $T^{\tau_{\Omega_0}(x)}(x)\in\interior (\Omega_0)$. One only need to apply the proof of Theorem~\ref{escape.balls.exponential} to the neighborhoods $T^{\tau_{\Omega_0}(x)}(B_r(x))$ of $T^{\tau_{\Omega_0}(x)}(x)\in \Omega_0$. Note that these neighborhoods are no longer balls, but the approximation argument in Theorem~\ref{escape.balls.exponential} can be adopted with minor modification. 
\end{remark}

\begin{proof}[Proof of Theorem~\ref{t.tower}]
Since $R$ is the first return map, $T$ and $F$ are conjugate. Below we will often interchange these two maps, and think of $\mu$ as a measure on the tower.
Recall that $\hat{T} = T^R$ is the induced map on $\Omega_0$, which preserves an invariant measure $\hat\mu$. Let $\hat\cA$ be the partition of $\Omega_0$ into $\Omega_{0,i}$'s. It is well known that the induced system $(\Omega_0, \hat{T}, \hat\cA,  \hat\mu)$ is exponentially $\phi$-mixing (see, for example, Lemma 2.4(b) in~\cite{MN05}). Moreover, the non-uniformly expanding condition (iii) guarantees that $\diam \hat\cA^n \lesssim \kappa^{n}$ (note that the partition $\hat\cA^n$ are defined using $\hat{T}$).

Furthermore, the invariant measure $\hat\mu$ is absolutely continuous with respect to the reference measure $\hat\nu$, where the density is indeed H\"older continuous. Since $\hat\nu = m|_{\Omega_0}$ is the Lebesgue measure on $\Omega_0$, conditions (iii) and  (iv) of Theorem~\ref{escape.balls.exponential} are satisfies. For $x\in \Omega_0, U\subset \Omega_0$, write 
$$
\hat\tau_U(x)=\inf\{j\ge1: \hat{T}^jx\in U\},
$$
$$
\hat\rho_U=\lim_{t\to\infty}\frac1t|\log\hat\mu(\hat\tau_U>t)|,
$$
and
$$
\hat\rho(x) = \lim_{r\to 0} \frac{\hat\rho(B_r(x))}{\hat\mu(B_r(x))}
$$
for the local escape rate of the induced system, we have:
\begin{proposition}\label{p.1}
Let $\hat{T}$ be the induced map on $\Omega_0$. Then
$$
\hat\rho(x)=\begin{cases}
1&\mbox{if $x$ is non-periodic}\\
1-\hat\vartheta&\mbox{if $x$ is periodic}\end{cases}
$$
for every $x\in\interior\Omega_0$ provided the limit 
$\hat\vartheta=\lim_r\frac{\hat\mu(\hat{T}^{-\hat m}B_r(x)\cap B_r(x))}{\hat\mu(B_r(x))}<\frac12$ exists in the case when $x$ is a periodic
point of $\hat{T}$ with minimal period $\hat m$, for which $D\hat{T}^{\hat m}(x)$ has no generalised eigenvectors.
\end{proposition}

From now on we will assume that $x\in\interior\Omega_0$ is non-periodic. For each $y\in M$ with $\tau_{\Omega_0}(y)<\infty$, we take $y_0$ to be the unique point in $\Omega_0$ with
$$
y = T^{m(y)}(y_0)
$$
with $0\le m(y)<R(y_0)$. In particular, if $y\in\Omega_0$ we take $y=y_0$ and $m(y)=0$.
Then we have
\begin{equation}\label{e.0}
\tau_{B_r(x)}(y) =-m(y)+\sum_{j=0}^{\hat\tau_{B_r(x)}(y_0)-1} R(\hat{T}^j(y_0)).
\end{equation}
By the Birkhoff ergodic theorem on $(\Omega_0,\hat{T},\hat\mu)$, we see that
$$
\frac{1}{n}\sum_{j=0}^{n-1}R(\hat{T}^jy_0)\to \int_{\Omega_0} R(y)\,d\hat\mu(y) = \frac{1}{\mu(\Omega_0)},
$$
where we apply the Kac's formula on the last equality and use the fact that $\mu$ is the lift of $\hat\mu$. Motivated by the large deviation estimate on the space of H\"older functions, for every $\varepsilon>0$ and $k\in\mathbb{N}$, we define the set
$$
B_{\varepsilon, k} = \left\{y\in \Omega_0:  \left|\frac{1}{n}\sum_{j=0}^{n-1}R(\hat{T}^jy_0) - \frac{1}{\mu(\Omega_0)} \right|>\varepsilon \mbox{ for some } n\ge k\right\},
$$
then it is known that (see, for example, \cite{BDT} Appendix B):
$$
\hat\mu(B_{\varepsilon,k}) \le C_\varepsilon e^{-c_\varepsilon k}
$$
for some constants $C_\varepsilon,c_\varepsilon>0$ depending on $\varepsilon$.

On the other hand, since the return time function $R$ has exponential tail, we get, for each $\varepsilon>0$ and $t$ large enough,
$$
\mu(y\in M:m(y)>\varepsilon t)\le \mu(y_0\in\Omega_0 :R(y_0)>\varepsilon t)\lesssim\lambda^{\varepsilon t}.
$$

To simplify notation, we introduce the set $$A_t=\left\{ y:m(y)<\varepsilon t,\sum_{j=0}^{\hat\tau_{B_r(x)}(y_0)-1} R(\hat{T}^j(y_0))>(1+\varepsilon)t\right\}
\cap B^c_{\varepsilon,k}.$$ 

Combine~\eqref{e.0} with the previous  estimates on $B_{\varepsilon,k}$ and $\{y:m(y)>\varepsilon t\}$, with $k=t(1+\varepsilon)$ we get
\begin{equation}\label{e.1}
\left|\mu(\tau_{B_r(x)}>t) - \mu(A_t)\right|
\lesssim \lambda^{\varepsilon t}+e^{-c_\varepsilon k}.
\end{equation}

Note that the set $A_t$ contains 
$$
A_t^- = \left\{y:m(y)<\varepsilon t, \hat\tau_{B_r(x)}(y_0)>\frac{(1+\varepsilon)t}{\mu^{-1}(\Omega_0)-\varepsilon}\right\},
$$
and is contained in 
$$
A_t^+ = \left\{y:m(y)<\varepsilon t, \hat\tau_{B_r(x)}(y_0)>\frac{(1+\varepsilon)t}{\mu^{-1}(\Omega_0)+\varepsilon}\right\}.
$$
Now we are left to estimate $\mu(A_t^\pm)$. Since $\mu$ is the lift of $\hat\mu$, we have
\begin{align}\label{e.2}
\mu(A^\pm_t) = &\frac{1}{\hat\mu(R)}\sum_{j=0}^{\infty}\sum_{i=0}^{\min(\varepsilon t, R_j)}\hat\mu(T^{-i}A^\pm_t\cap \Omega_{0,i})\\\nonumber
=&\mu(\Omega_0)(1+\mathcal{O}(\varepsilon t))\mu_0(\hat{A}^\pm_t),
\end{align}
where
$$
\hat{A}^\pm_t = \left\{y_0\in\Omega_0: \hat\tau_{B_r(x)}(y_0)>\frac{(1+\varepsilon)t}{\mu^{-1}(\Omega_0)\pm\varepsilon}\right\}.
$$
By Proposition~\ref{p.1} and the observation that $\hat\mu(B_r(x)) \mu(\Omega_0)=\mu(B_r(x))$, we have 
$$
\lim_{r\to 0}\lim_{t\to\infty}\frac{1}{t\mu(B_r(x))}|\log\hat\mu(\hat{A}^\pm_t)| = \frac{(1+\varepsilon)}{1\pm\varepsilon\mu(\Omega_0)}.
$$
By~\eqref{e.2}, we get that 
$$
\lim_{r\to 0}\lim_{t\to\infty}\frac{1}{t\mu(B_r(x))}|\log\mu(\hat{A}^\pm_t)| = \frac{(1+\varepsilon)}{1\pm\varepsilon\mu(\Omega_0)}.
$$

For each $\varepsilon>0$ we can take $r$ small enough, such that 
$$
\frac{1+\varepsilon}{1\pm\varepsilon \mu(\Omega_0)}\mu(B_r(x)) < \min\{\varepsilon| \log\lambda|,c_\varepsilon(1+\varepsilon)\}. 
$$
It then follows that the right-hand-side of \eqref{e.1} is of order $o(\mu(A^\pm_t)).$ We thus obtain
$$
\rho(x)=\lim_{r\to 0}\lim_{t\to\infty}\frac{1}{t\mu(B_r(x))}|\log\mu(\tau_{B_r(x)}>t)| \in \left(\frac{(1+\varepsilon)}{1+\varepsilon\mu(\Omega_0)},\frac{(1+\varepsilon)}{1-\varepsilon\mu(\Omega_0)}\right)
$$
for every $\varepsilon>0$. This shows that $\rho(x)=1$ if $x$ is non-periodic.

When $x$ is periodic for the map $T$ with period $m$, it is also periodic for $\hat{T}$ with periodic $\hat m$, where $\hat m$ is less than $m$. In this case it is well known that $\vartheta = \hat\vartheta=e^{\sum_{j=0}^{m-1}\phi(T^j(x))}$,
where the potential function $\phi$ is given by $\phi(x)=-\log\det DT(x)$.  Then the same proof as before shows that $\rho(x) = 1-\hat\vartheta = 1-\vartheta$.

\end{proof}
 \section{Examples}
\subsection{Subshift of finite type}

Let us consider the special case when $\Omega$ is a subshift of 
finite type over a finite alphabet $\mathcal{A}$ and with 
transition matrix $G$ which we assume is irreducible. Then 
$\Omega=\left\{x\in\mathcal{A}^{\mathbb{N}_0}: G_{x_i,x_{i+1}}=1\forall i\ge0\right\}$
and $T:\Omega\circlearrowleft$ is the left shift map.
A function $f:\Omega\to\mathbb{R}$ is H\"older continuous if 
$|f|_\alpha=\sup_n\alpha^{-n}\var_nf$ is finite, where 
$\var_nf=\sup_{A\in\mathcal{A}^n}\sup_{x,y\in A}|f(x)-f(y)|$ is the $n$-variation of $f$.
The norm $\|f\|=|f|_\infty+|f|_\alpha$ then makes the functions space 
$C^\alpha(\Omega)=\{f: \|f\|<\infty\}$ a Banach space. A potential function
$f\in C^\alpha$ then allows one to define the transfer operator 
$\mathcal{L}:C^\alpha\circlearrowleft$ by 
$\mathcal{L}\varphi(x)=\sum_{y\in\sigma^{-1}x}e^{f(y)}\varphi(y)$
where $\sigma:\Omega\circlearrowleft$ is the shift map on $\Omega$ given
by $(\sigma x)_i=x_{i+1}\forall i$.
Its dominant eigenvalue is simple and positive real. The equilibrium state $\mu$
is then given by $\mu=h\nu$ where $h$ is its eigenfunction and $\nu$ its eigenfunctional
 (normalised so that $\nu(h)=1$) (see e.g.\ \cite{Bow,Wal}). 
On $\Omega$ one has the usual metric which defined by $d(x,y)=\alpha^{n(x,y)}$, 
 where $n(x,y)=\min\{|i|: x_i\not=y_i\}$. Sets 
 $U_n\in\sigma(\mathcal{A}^n)$ are closed-open sets. 

\begin{thm}\label{escape.subshift}
Let $\Omega$ be a subshift of finite type with alphabet $\mathcal{A}$
 and $\mu$ and equilibrium state for  a H\"older continuous function $f$.
 
 Let  $U_n\in\sigma(\mathcal{A}^n)$, $n=1,2,\dots,$ be so that
$U_{n+1}\subset U_n$, $\bigcap_{n}U_n=\{x\}$ and $\diam\,U_n\le\eta^n$ for 
some $\eta<1$ and all $n$ large enough.

If 
$$
\left\{
\begin{array}{lcl}
\mbox{either} & n^{-\gamma''}\lesssim\mu(U_n)\lesssim n^{-\gamma'}&\mbox{for some 
$1<\gamma'<\gamma''$}\\
\mbox{or}& \xi_1^n\lesssim\mu(U_n)\lesssim\xi_2^n&\mbox{for some $0<\xi_1<\xi_2<1$}
\end{array}
\right.,
$$
then
$$
\rho(x)=\left\{\begin{array}{ll}1&\mbox{ if $x$ is not periodic}\\
1-e^{f^m(x)-mP(f)}&\mbox{ if $x$ is periodic with minimal period $m$  provided $e^{f^m(x)-mP(f)}<\frac12$}
\end{array}\right.,
$$
where $P(f)$ is the pressure of $f$ and where 
$f^m=f+f\circ\sigma+\cdots+f\circ\sigma^{m-1}$
 is the $m$th ergodic sum of $f$.
\end{thm}

\noindent In particular the limit defining $\vartheta(x)$ exists for all periodic
points $x$. Let us first prove the following lemma:

 \begin{lemma}\label{diameter.to.zero}
 Let $\Omega$ be a subshift over a finite alphabet and $U_n\in\sigma(\mathcal{A}^n), n=1,2,\dots$,
be so that $U_{n+1}\subset U_n\forall n$ and  $\{x\}=\bigcap_nU_n$ for a periodic point
 $x$ with minimal period $m$.

 Then for every $u\in\mathbb{N}$ there exists an $N_u$ so that 
 $U_{n}\subset A_{mu}(x)$ for all $n\ge N_u$.
  \end{lemma}
 
 \begin{proof} Let $x$ be periodic with minimal period $m$. By the nested assumption
 and the intersection property we can find
 for every $y\not=x$ an $N(y)<\infty$ so that $y\not\in U_n\forall n\ge N(y)$.
 The function $N$ is continuous (in fact locally constant) and 
 since $A_{um}(x)\in\mathcal{A}^{um}$ is a closed-open set,
 $N_u=\sup_{y\not\in A_{um}(x)}N(y)$ is finite. Hence 
 $y\not\in U_{n,u}\forall n\ge N_u$  for all $y\not\in A_{um}(x)$.
  \end{proof}
  
 \begin{lemma}\label{vartheta.subshift}
 Let $\mu=h\nu$ be the equilibrium state for the H\"older continuous 
 function $f$, where $\nu$ is the conformal measure and $h\in C^\alpha$ 
 the associated density.
 
  Let $x\in\Omega$ be a periodic point with minimal period $m$
  and assume $U_n\in\sigma(\mathcal{A}^n)$, $n=1,2,\dots$, are so that 
  $U_{n+1}\subset U_n\,\forall n$  and $\{x\}=\bigcap_nU_n$. Then the limit 
$$
\vartheta=\lim_{n\to\infty}\frac{\mu(T^{-m}U_n\cap U_{n})}{\mu(U_n)}=e^{f^m(x)-mP(f)}
$$
exists. 
\end{lemma}

\begin{proof} Without loss of generality we  can assume that $P(f)=0$,
 otherwise we replace $f$ by $f-P(f)$ which has zero pressure
 and has the same equilibrium state.
Note that for an $n$-cylinder $A_n\in\mathcal{A}^n$ one has
 $\nu(A_n)=\int \chi_{A_n}(x)\,d\nu(x)=\int e^{f^n(A_nx)}\,d\nu(x)$.
We treated $A_n$ as a word of length $n$ and with $A_nx$ we
mean the concatenation of $A_n$ with $x$ (or the point $T^{-n}x\cap A_n$).
 By the same token we obtain for 
an $n+m$ word $A_{n+m}\in\mathcal{A}^{n+m}$ that
$$
\nu(A_{n+m})=\int e^{f^{n+m}(A_{n+m}x)}\,d\nu(x)
=e^{f^m(A_{n+m}x)+\mathcal{O}(\alpha^n)}\int e^{f^{n}(A_{n}y)}\,d\nu(y)
$$
for any $x$ for which $A_{n+m}x\in\Omega$. 
This follows from the regularity of $f$ which implies that
$f^{n+m}(A_{n+m}y)=f^{m}(A_{n+m}x)+\mathcal{O}(\alpha^n)+f^{n}(A_{n}y)$.
Consequently
$$
\nu(A_{n+m})=e^{f^m(A_{n+m}x)+\mathcal{O}(\alpha^n)}\nu(A_n).
$$
Let $x$ be periodic with minimal period $m$
and $U_n\in\sigma(\mathcal{A}^n)$, $n=1,2,\dots$, be a neighbourhood system of $x$. 
For points $x_k$ one then writes $U_n=\bigcup_kA_n(x_k)$
(disjoint union) and obtains
$$
U_{n,1}=A_m(x)\cap \bigcup_kA_n(x_k)
=\bigcup_kA_{n+m}(\tilde{x}_k),
$$
where $\tilde{x}_k$ denotes the unique point $T^{-m}x_k\cap A_m(x)$.
According to Lemma~\ref{diameter.to.zero} 
for every $u\in\mathbb{N}$ there exists $N_u$ so that 
$U_n\subset A_{um}(x)\forall n\ge N_u$. This implies that $x_k\in A_{um}(x)$ for all $k$.
Then
\begin{eqnarray*}
\mu(A_{n+m}(\tilde{x}_k))&=&\nu(A_{n+m}(\tilde{x}_k))(h(\tilde{x}_k)+\mathcal{O}(\alpha^{n+m}))\\
&=&e^{f^m(\tilde{x}_k)+\mathcal{O}(\alpha^n)}\nu(A_{n}(x_k))(h(x_k)+\mathcal{O}(\alpha^{um}))\\
&=&e^{f^m(x)+\mathcal{O}(\alpha^{(u-1)m})}\mu(A_{n}(x))
\end{eqnarray*}
using the periodicity of $x$ and the fact that $h\in C^\alpha$  implies
$h(\tilde{x}_k)=h(x)+\mathcal{O}(\alpha^{(u+1)m})=h(x_k)+\mathcal{O}(\alpha^{um})$.
Hence
$$
\frac{\mu(U_{n,1})}{\mu(U_n)}
=\frac{\sum_ke^{f^m(x)+\mathcal{O}(\alpha^{(u-1)m})}\mu(A_{n}(x))}{\sum_k\mu(A_{n}(x))}
=e^{f^m(x)+\mathcal{O}(\alpha^{(u-1)m})}
$$
and consequently the limit 
$$
\vartheta=\lim_{n\to\infty}\frac{\mu(U_{n,1})}{\mu(U_n)}
=\lim_{u\to\infty}e^{f^m(x)+\mathcal{O}(\alpha^{(u-1)m})}
=e^{f^m(x)}
$$
exists. 
\end{proof}

\begin{proof}[Proof of Theorem~\ref{escape.subshift}.]
Let us first point out that if  $x$ is a periodic point with minimal period $m$
then by Lemma~\ref{vartheta.subshift} the limit $\vartheta=e^{f^m(x)-mP(f)}$ exists.

It remains to verify Properties~(4) and~(5).
For the shift space $\Omega$ with the standard metric (which is given by 
 $d(x,y)=a^{n(x,y)}$ for $x,y\in\Omega$, where $n(x,y)=\min\{j\ge0: x_j\not=y_j\}$ and $a\in(0,1)$ is arbitrary),
 one has
$\diam\, A=\alpha^n\;\forall A\in\mathcal{A}^n$. Since by assumption
$\diam\,U_n\le \eta^n$ for all $n$ large enough, we conclude that
$U_n\subset A_j(x)$ for all $j\le Jn$ where $J=\frac{\log\eta}{\log\alpha}$
is independent of $n$.
Since  $\mu$ is in fact $\psi$-mixing (see e.g.~\cite{Bow})
at an exponential rate, 
it is also left and right $\phi$-mixing at an exponential rate. 
If $x$ is periodic with period $m$, 
then $U_{n,u}^j=A_j(U_{n,u})\subset A_j(x)$ 
for $j\le J(n+um)$. Since the measure of $j$-cylinders
decays exponentially we obtain that $\mu(U_{n,u}^j)\le\gamma^j$
for some $\gamma<1$ and all $j\le J(n+um)$.
The same argument also yields Property~(5).

The statement of the theorem now follows from an application of
Theorem~\ref{escape.phi}~(I) in the case when $\mu(U_n)$ decays 
polynomially and from Theorem~\ref{escape.phi}~(II)
in the case of exponential decay of $\mu(U_n)$.
\end{proof}


\subsection{Gibbs-Markov maps}
We consider a Gibbs-Markov map $T$ on a Lebesgue space $(X, \mu)$. Recall that a map $T$ is called {\em Markov} if there is a countable measurable partition $\cA$ on $X$ with $\mu(A)>0$ for all $A\in \cA$, such that for all $A\in \cA$, $T(A)$ is injective and can be written as a union of elements in $\cA$. Write $\cA^n=\bigvee_{j=0}^{n-1}T^{-j}\cA$ as before, it is also assumed that $\cA$ is (one-sided) generating.

Fix any $\lambda\in(0,1)$ and define the metric $d_\lambda$ on $X$ by $d_\lambda(x,y) = \lambda^{s(x,y)}$, where $s(x,y)$ is the largest positive integer $n$ such that $x,y$ lie in the same $n$-cylinder. Define the Jacobian $g=JT^{-1}=\frac{d\mu}{d\mu\circ T}$ and $g_k = g\cdot g\circ T \cdots g\circ T^{k-1}$.

The map $T$ is called {\em Gibbs-Markov} if it preserves the measure $\mu$, and also satisfies the following two assumptions:\\
(i) The big image property: there exists $C>0$ such that $\mu(T(A))>C$ for all $A\in \cA$.\\
(ii) Distortion: $\log g|_A$ is Lipschitz for all $A\in\cA$.

For example, if $T$ is modeled by Young's tower with a base $\Omega_0$, then the return map $\hat{T} = T^R:\Omega_0\to\Omega_0$ is a Gibbs-Markov map with respect to the invariant measure $\mu|_{\Omega_0}=(h\nu)|_{\Omega_0}$ and the partition $\{\Omega_{0,i}\}$, since $\hat{T}(\Omega_{0,i}) =\Omega_0$.

In view of (i) and (ii), there exists a constant $D>1$ such that for all $x,y$ in the same $n$-cylinder, we have the following distortion bound:
$$
\left|\frac{g_n(x)}{g_n(y)}-1\right|\le D d_\lambda(T^nx,T^ny),
$$
and the Gibbs property:
$$
D^{-1}\le \frac{\mu(A_n(x))}{g_n(x)}\le D.
$$
It is well known (see Lemma 2.4(b) in~\cite{MN05}) that Gibbs-Markov systems are exponentially $\phi$-mixing. Therefore we have the following corollary of Theorem~\ref{escape.phi}:

\begin{thm}
	Let $T$ be a Gibbs-Markov map on $(X,\mu)$. With $U_n = A_n(x)$, we have	
	$$
	\rho(x)=\left\{\begin{array}{ll}1&\mbox{ if $x$ is not periodic}\\
	1-g_m(x)&\mbox{ if $x$ is periodic with minimal period $m$  provided $g_m(x)<\frac12$}
	\end{array}\right.
	$$
\end{thm}
\begin{proof}
	Recall that the $n$-cylinders $A_n(x)$ have exponentially small measure and satisfy the conditions (1) to (5) of the adapted neighborhood system. Then Theorem~\ref{escape.phi} gives
	$\rho(x) = 1-\vartheta(x)$, provided that $\vartheta(x)<\frac12$. 
	
	We are left to check that $\vartheta(x) = g_m(x)$. The proof is essentially the same as Lemma~\ref{vartheta.subshift} with $f = \log g$ (thus $e^{f^m(x)} = g_m(x)$). Also note that $P(\log g)=\log\sigma(\mathcal{L})=0$, where $\mathcal{L}$ is the transfer operator with respect to $f$ and $\sigma(\mathcal{L})$ is the spectral radius of $\mathcal{L}$; see for example \cite[Corollary 2.3]{MN05}.
\end{proof}

\subsection{Interval maps} As an example we consider interval maps modeled by Young's tower. Examples include uniform expanding piecewise $C^2$-map of the 
unit interval $I=[0,1]$ with the Markov property and certain unimodal maps.
Let $T:I\circlearrowleft$ be such a map, then it has an absolutely 
continuous invariant measure $\mu$ which has a positive density
$h$ with respect to Lebesgue measure $\lambda$ (see e.g.~\cite{GB} and~\cite{Y3}).
The following result was proven in~\cite{BY11} Theorem~4.6.1 for the doubling
map $T(x)=2x\mod 1$ on the unit interval. There the periodic points $x$ of
minimal period $m$ have dyadic expansion $x=0.\overline{x_1x_2\cdots x_m}$,
$x_i\in\{0,1\}$ and where the $m$-word $x_1\cdots x_m$ is not generated by
repeating a shorter word. Then it was shown that $\rho(x)=1-2^{-m}$. A similar result was also obtained  in~\cite{BDT} for general interval maps using inducing scheme, provided that the return time function $R$ has exponential tail and satisfies an exponential large deviation estimate.

\begin{thm}
Let $T$ be an $C^{1+\alpha}$ map on the unit interval which can be modeled by Young's tower with exponential tail, i.e., $\lambda(\hat{R}>n) \lesssim \eta^n$ for some $\eta\in(0,1)$. Here $\lambda$ is the Lebesgue measure on $[0,1]$. Let $\mu$ be absolutely 
continuous invariant measure with density $h(x)$ w.r.t. $\lambda$. 

The local escape rate is then
$$
\rho(x)
=\lim_{r\to0}\frac{\rho_{B_r(x)}}{\mu(B_r(x))}
=\begin{cases}1&\mbox{if $x$ is non-periodic,}\\
1-\frac1{|(T^m)'(x)|}&\mbox{if $x$ is periodic with minimal period $m$  } \\ &
\mbox{with $h(x)>0$ and provided $|(T^m)'(x)|>2$.}\end{cases}
$$
\end{thm}

\noindent Note that the expression for periodic points matches 
Theorem~\ref{escape.subshift} if one sets $f=-\log|T'|$. This 
function has zero pressure, i.e. $P(f)=0$. The absolutely continuous
measure $\mu$ is then a Gibbs state for $f$.

\begin{proof} This follows immediately from Theorem~\ref{t.tower}. We only have to determine $\vartheta$.

Since $\mu$ is absolute continuous with respect to the Lebesgue measure $\lambda$, we have that $D=1$. 
The assumption that at periodic points $DT^m(x)=(T^m)'(x)$ has no generalised eigenvectors is clearly 
satisfied as we are in dimension one.

 To prove the existence of the limit $\vartheta$ for periodic points,
 let $x$ be a periodic point with minimal period $m$, such that $h(x)>0$. 
 Then $\mu(B_r(x))=\lambda(B_r(x))(h(x)+o(1))$ as $r\to0$, 
 For all small enough $r>0$, $T^{-m}B_r(x)\cap B_r(x)=T^{-m}B_r(x)\subset B_r(x)$
 and therefore 
 $$
 \mu(T^{-m}B_r(x)\cap B_r(x))=\mu(T^{-m}B_r(x))=\frac1{|(T^m)'(x)|}\lambda(B_r(x))(h(x)+o(1))
 $$
 as $r\to0$. Thus
 $$
 \vartheta(x)=\lim_{n\to\infty}\frac{ \mu(T^{-m}B_r(x)\cap B_r(x))}{\mu(B_r(x))}
 =\frac1{|(T^m)'(x)|}.
 $$

 \end{proof}

\subsection{Conformal repeller} This example was covered in~\cite{FP} and 
deals with $C^1$-maps $T$ on Riemannian manifolds $M$. A conformal repeller
is then a maximal compact set $\Omega\subset M$ so that $T$ acts conformally
on $\Omega$ and is expanding, that is there exists a $\beta>1$ so that
 $|DT^kv|\ge \beta^k$ for all large enough $k$ and all $v\in T_xM\forall x\in\Omega$.

\begin{thm}\label{escape.conformal.repeller}
Let $\Omega\subset M$ be a conformal repeller for the $C^1$-map $T:M\circlearrowleft$
and let $\mu$ be an equilibrium state for a H\"older continuous potential 
$f:\Omega\to\mathbb{R}$.

The local escape rate for metric balls is then
$$
\rho(x)=\lim_{r\to0}\frac{\rho_{B_r(x)}}{\mu(B_r(x))}=\begin{cases}1&\mbox{if $x$ is non-periodic}\\
1-e^{f^m(x)-mP(f)}&\mbox{if $x$ is periodic with minimal period $m$}\\&
\mbox{provided $e^{f^m(x)-mP(f)}<\frac12$} \end{cases}
$$
\end{thm}

\begin{proof}  We use the fact that  $\Omega$ allows Markov partitions $\mathcal{A}$ 
of arbitrarily small diameter. Let $\mathcal{A}$ be a generating Markov partition and
we verify the assumptions of Theorem~\ref{escape.balls}:\\
 (i) Is satisfied because the equilibrium state $\mu$
 is $\psi$-mixing with respect to the partition $\mathcal{A}$ and therefore
 also right and left $\phi$-mixing where $\psi$ (and therefore $\phi$) decays
  exponentially fast.\\
 (ii) This follows from expansiveness: $\diam \,\mathcal{A}^n=\mathcal{O}(\eta^n)$
 with $\eta=\frac1\beta<1$ and $\zeta=1$. \\
 (iii) Is satisfied for any $w>1$ as $\mu$ is diametrically regular~\cite{PW97} and
 thus also has the annular decay property~\cite{Bu99}. This yields 
 $\frac{\mu(B_{r+r^w}(x)\setminus B_r(x))}{\mu(B_r(x))}=\mathcal{O}(r^{(w-1)\delta})\to0$
 for some $\delta>0$. \\
 (iv) Is satisfied with some $d>0$ by~\cite{FP} Lemma~6.4.\\
  The existence of the limit $\vartheta$ for periodic points follows from 
 Lemma~\ref{vartheta.subshift}.
  
In order to verify Property~(5) let us note that conformality
 is that $DT(x)=a(x)I_x$ where $a(x)\ge\beta$ is the dilation function and $I_x$ is 
 an isometry. If $x$ is periodic with minimal period $m$, 
 then $DT^m(x)=a^m(x)I_x^m$ maps the tangent space $T_xM$ to itself
 with $x$ fixed point ($a^m(x)=a(x)a(Tx)\cdots a(T^{m-1}x)$). 
 In particular $A=DT^m(x)$ has no generalised eigenvector and moreover for all $k$
one has $B'_r(0)\cap A^{-k}B'_r(0) =\bigcap_{j=0}^kA^{-j}B'_r(0)=(a^{m}(x))^{-k}B'_r(0)$
  is just the ball  $B'_r(0)$ scaled by the factor $(a^{m}(x))^{-k}$.
  
  Let $w\in(1,2)$ and as in  Theorem~\ref{escape.balls} we obtain
 $\diam\,\mathcal{A}^n<r^w$ provided $n> w\frac{\log r}{\log\eta}$.

With the exponential map $\Phi_x$ at $x$, we get $B_r(x)=\Phi_x(B'_r(0))$
and, as before, put  $B^{(k)}_r(x)=\bigcap_{j=0}^kT^{-jm}B_r(x)$ and 
 $B'^{(k)}_r(x)=\bigcap_{j=0}^kA^{-j}B'_r(x)$.
Since $d(T^{jm}y,\Phi_xA^j\Phi_x^{-1}y)\lesssim r^2\:\forall y\in B^{(k)}_r(x)$
we conclude
$$
\mathcal{D}=\Phi_x(B'^{(k)}_r(0))\triangle B^{(k)}_r(0)
\subset B_{r'}\!\left(\Phi_x\!\left(\partial B'^{(k)}_r(0)\right)\right)
\subset B_{r''}(\partial B_{(a^m(x))^{-k}r}(x))
$$
for some $r''\lesssim r'\lesssim r^2$.
Naturally we need $ ((a^m(x))^{-k}r)^w\gtrsim r'$ which is achieved since 
 $r^w\ge\eta^{n}\ge\diam\,\mathcal{A}^n$ implies  that $r^2\le(a^m(x))^{-Jw\frac{n}m}r^w$ 
 is satisfied.
 Consequently we must choose 
 $$
 J<m\frac{|\log\eta|}{\log a^m(x)}\frac{2-w}{w^2}.
 $$
  We  therefore obtain by~(iii) above: 
\begin{eqnarray*}
\frac{\mu(\mathcal{D})}{\mu(\Phi_x(B'^{(k)}_r(0)))}
&\le&\frac{\mu(B_{(a^m(x))^{-k}r+((a^m(x))^{-k}r)^w}\setminus B_{(a^m(x))^{-k}r-((a^m(x))^{-k}r)^w}(x_j))}
{\mu(B_{(a^m(x))^{-k}r}(x_j))}\\
&\lesssim&((a^m(x))^{-k}r)^{(w-1)\delta}
\longrightarrow0
\end{eqnarray*}
 as $r$ goes to zero.
 
We have thus shown that if we have indices $0=i_0<i_1<\cdots<i_k\le Jn$ which are multiples of $m$,
then
$$
\mu\!\left(\bigcap_{j=0}^kA^{-i_j/m}B'_r(0)\right)
=\mu\!\left(\bigcap_{j=0}^kT^{-i_j}B_r(x)\right)\!\left(1
+o(1)\right)
$$
as $r\to0$. We have thus verified Property~(5) for the adapted neighbourhood system.

 The result now follows from Theorem~\ref{escape.balls}.
\end{proof}

\end{document}